\documentclass[reqno,12pt]{amsart} 

\usepackage{amssymb}

\usepackage{bookmark}
\usepackage[left=1.5in, right=1.5in, bottom=1.5in]{geometry}
\newtheorem{theorem}{Theorem}[section]
\newtheorem{lemma}[theorem]{Lemma}
\newtheorem{proposition}[theorem]{Proposition}
\newtheorem{corollary}[theorem]{Corollary}
\theoremstyle{definition}

\theoremstyle{remark}
\newtheorem{remark}[theorem]{Remark}

\numberwithin{equation}{section}

\newcommand{\norm}[1]{\lVert#1\rVert}

\newcommand{\Ag}[1]{\langle#1\rangle}

\newcommand{\cL}{\mathcal{L}}

\newcommand{\R}{\mathbb{R}}

\newcommand{\Z}{\mathbb{Z}}
\newcommand{\C}{\mathbb{C}}
\newcommand{\T}{\mathbb{T}}

\newcommand{\e}{\varepsilon}

\newcommand{\txt}[1]{\text{#1}}

\begin{document}

\title[]{First-order expansions for eigenvalues and eigenfunctions in periodic homogenization}


\author{Jinping Zhuge}

\address{Department of Math, University of Kentucky, Lexington, KY, 40506, USA.}
\email{jinping.zhuge@uky.edu}
\thanks{The author is supported in part by National Science Foundation grant DMS-1600520.}


\begin{abstract}
	For a family of elliptic operators with periodically oscillating coefficients, $-\text{div}( A(\cdot/\e) \nabla) $ with tiny $\e>0$, we comprehensively study the first-order expansions of eigenvalues and eigenfunctions (eigenspaces) for both Dirichlet and Neumann problems in bounded, smooth and strictly convex domains (or more general domains of finite type). A new first-order correction term is introduced to derive the expansion of eigenfunctions in $L^2$ or $H^1_{\txt{loc}}$. Our results rely on the recent progress on the homogenization of boundary layer problems.
\end{abstract}

\maketitle
\pagestyle{plain}
\section{Introduction}
This paper concerns with the first-order expansions of eigenvalues and eigenfunctions (eigenspaces) for a family of elliptic operators with rapidly oscillating coefficients. Precisely, we consider
\begin{equation}\label{eq.Le}
\cL_\e = -\txt{div} (A(x/\e) \nabla) = - \frac{\partial}{\partial x_i} \bigg\{ a^{\alpha\beta}_{ij} \Big( \frac{x}{\e}\Big) \frac{\partial}{\partial x_j}\bigg\},
\end{equation}
(Einstein's summation convention will be used throughout) where $1\le i,j\le d, 1\le \alpha,\beta \le m$ with dimension $ d\ge 3$. The coefficient matrix $A = (a^{\alpha\beta}_{ij})$ satisfies the following {\em standard assumptions}:
\begin{itemize}
	\item Ellipticity: there exists $\Lambda>0$ such that
	\begin{equation}
	\Lambda^{-1} |\xi|^2 \le a^{\alpha\beta}_{ij} \xi_i^\alpha \xi_j^\beta \le \Lambda |\xi|^2, \quad \txt{for any } \xi = (\xi_i^\alpha) \in \R^{m\times d}.
	\end{equation}
	\item Periodicity: $A(y+z) = A(y),$ for any $z \in \Z^d $ and $y\in \R^d.$
	\item Regularity: $a^{\alpha\beta}_{ij} \in C^\infty(\R^d)$.
	\item Symmetry: $A^* = A$ (i.e., $a^{\alpha\beta}_{ij} = a^{\beta\alpha}_{ji}$).
\end{itemize}
Throughout this paper, we assume that the domain $\Omega$ is bounded and smooth. We refer to a recent excellent book \cite{ShenBook} for general theory of periodic homogenization.

The asymptotic analysis of the eigenvalues and eigenfunctions is an important and interesting problem with many applications including the homogenization of heat equations and wave equations (low-frequency vibration) in composite materials with periodic microstructure; see \cite{Os75,Kesavan1,Kesavan2,SV93,AC98,Prange13,KS13}.  To describe our main concern of this paper, we concentrate on the Dirichlet problem. Let $\{ \lambda_{\e,k} \}_{k\ge 1}$ denote the sequence of Dirichlet eigenvalues in an increasing order for $\cL_\e$ in $\Omega$ and let $ \phi_{\e,k}$ be the normalized Dirichlet eigenfunction corresponding to $\lambda_{\e,k}$, i.e., $\phi_{\e,k}\in H_0^1(\Omega;\R^m), \norm{\phi_{\e,k}}_{L^2(\Omega)} = 1$ and $\cL_\e \phi_{\e,k} = \lambda_{\e,k} \phi_{\e,k}$. Let $\{ \lambda_{0,k} \}_{k\ge 1}$ denote the sequence of Dirichlet eigenvalues in an increasing order for the homogenized operator $\cL_0$ in $\Omega$ and let $\{ \phi_{0,k} \}_{k\ge 1}$ be the corresponding normalized Dirichlet eigenfunctions for $\cL_0$.

It is well-known that for each $k\ge 1$ (see, e.g., \cite[Chapter 6.2]{ShenBook}),
\begin{equation}\label{cdn_lambda}
|\lambda_{\e,k} - \lambda_{0,k}| \le C_k \e.
\end{equation}
This is exactly the zero-order expansion for the Dirichlet eigenvalues $\lambda_{\e,k}$. However, the first-order expansion of eigenvalues is a more difficult problem as the higher-order rate of convergence in homogenization theory essentially involves PDEs with oscillating boundary data and the geometry of domains. In \cite{SV93} and \cite{MV97}, Vogelius and his collaborators attempted to study the asymptotic behavior of $(\lambda_{k,\e} - \lambda_{0,k})/\e$ as $\e \to 0$, and they showed that if $\Omega$ is a classical convex polygon with all sides having rational normal vectors, then the limit of $(\lambda_{k,\e} - \lambda_{0,k})/\e$ is not just one point, but rather a continuum of accumulation points. The lack of uniqueness of the limit is caused by the non-homogenization of the boundary layer problems (see (\ref{eq_Due})) in such domains. The homogenization of boundary layer problem was a longstanding open problem, and significant progress have been made recently in a series of papers \cite{GM11,GM12,AKMP16,ShenZhuge16,ShenZhuge17,Zhuge17}. The breakthrough was due to G\'{e}rard-Varet and Masmoudi's striking work \cite{GM12,GM11}, in which they showed that the Dirichlet boundary layer problem homogenizes with an explicit rate of convergence, provided additionally that $\Omega$ is a smooth, strictly convex domain or a convex polygon whose normal vectors satisfying a Diophantine condition (also referred as small divisor condition). Following by their work, Prange studied the first-order expansion of the Dirichlet eigenvalues in \cite{Prange13} for both strictly convex smooth domains and convex polygons with Diophantine normals.

To describe the main result of \cite{Prange13}, we let $\lambda_0 = \lambda_{0,L} = \lambda_{0,L+1} = \cdots =\lambda_{0, L+M-1}$ be a Dirichlet eigenvalue of $\cL_0$ with multiplicity $M\ge 1$ and let $\lambda_{\e,L+j}, 0\le j \le M-1$, be the Dirichlet eigenvalues of $\cL_\e$ that converge to $\lambda_0$. In \cite{Prange13}, Prange proved that if $\Omega\subset \R^2$ is bounded, smooth and strictly convex, then there exists some fixed constant $\theta$ such that
\begin{equation}\label{est.Prange}
\bigg| \bigg( \frac{1}{M} \sum_{j=0}^{M-1} \frac{1}{\lambda_{\e,L+j}} \bigg)^{-1} - \lambda_0 - \e \theta \bigg| \le C \e^{\frac{12}{11}-}.
\end{equation}
Here and after, we write $a\le C\e^{b-}$ to indicate that $a\le C_{\sigma}\e^{b-\sigma}$ for any $\sigma \in (0,b)$ with $C_\sigma$ depending on $\sigma$. Note that the first term of (\ref{est.Prange}) is the harmonic average of $\{ \lambda_{\e,L+j}: 0\le j\le M-1 \}$. 
The exponent $\frac{12}{11}$ comes from the rate of convergence for 2-dimensional boundary layer problem obtained in \cite{GM12}. More recently, G\'{e}rard-Varet and Masmoudi's result was improved to an almost optimal rate of convergence \cite{AKMP16,ShenZhuge16} and generalized to Neumann problem \cite{ShenZhuge16} and other type of domains \cite{Zhuge17}. As a consequence, we can easily extend Prange's result to higher dimensions ($d\ge 3$) as follows.
\begin{theorem}\label{thm_lambda}
	Let $A$ satisfy the standard assumptions and $\Omega$ be a bounded, smooth and strictly convex domain. Let $\lambda_0, \lambda_{\e,L+j}$  $(0\le j\le M-1)$ be the Dirichlet eigenvalues defined previously. Then there exist a constant $\theta$ independent of $\e$ such that for sufficiently small $\e > 0$
	\begin{equation}\label{est.lambe}
	|\bar{\lambda}_\e - \lambda_0 - \e \theta| \le C\e^{\frac{3}{2}-},
	\end{equation}
	where $\bar{\lambda}_\e = M^{-1} \sum_{j=0}^{M-1} \lambda_{\e,L+j}$, and $C$ depends only on $\lambda_0,A$ and $\Omega$.
\end{theorem}

We should point out that the rate in (\ref{est.lambe}) is $O(\e^{\frac{5}{4}-})$ for $d=2$, which improves (\ref{est.Prange}). The explicit formula for $\theta$ is given by (also see \cite{Prange13})
\begin{equation*}
\theta = -\frac{\lambda_0}{M} \Ag{K^{bl}\phi_{0,L+j}, \phi_{0,L+j}},
\end{equation*}
where $\Ag{\cdot,\cdot}$ denotes the inner product in $L^2(\Omega;\R^m)$, $\phi_{0,L+j}, 0\le j\le M-1$, are the eigenfunctions of $\cL_0$ corresponding to $\lambda_{0}$, and $K^{bl}$ is a linear operator naturally arising in the homogenization of boundary layers; see (\ref{def_K1}) and (\ref{eq.vblD}) for the definition. The exponent $\frac{3}{2}$ in (\ref{est.lambe}) seems optimal due to the optimality of the convergence rate for Dirichlet boundary layer problem in Theorem \ref{thm_blL2}. The proof of Theorem \ref{thm_lambda} follows from the same argument as \cite{Prange13} by using Osborn's theorem \cite{Os75}, yet by a simple observation, we replace the harmonic average in (\ref{est.Prange}) by the usual arithmetic average $\bar{\lambda}_\e$.

Now we turn to the main contribution of this paper, i.e., the first-order (two-scale) expansion of the eigenfunctions or eigenspaces, which is not known to the best of our knowledge. Recall that for $k\ge 1$ so that $\lambda_{0,k}$ is simple, one has (see \cite{Os75} or Lemma \ref{lem_EeEL2})
\begin{equation*}
\norm{\phi_{\e,k} - \phi_{0,k}}_{L^2(\Omega)} \le C_k\e.
\end{equation*}
Then, a natural question similar to eigenvalues arises: does $(\phi_{\e,k} - \phi_{0,k})/\e$ have a unique limit in some sense, as $\e \to 0$? To describe our result regarding this question, consider the Dirichlet problem
\begin{equation}\label{eq_Leue}
\left\{
\begin{aligned}
\cL_\e u_\e &= f &\quad & \txt{in } \Omega, \\
u_\e &= 0 &\quad & \txt{on } \partial\Omega.
\end{aligned}
\right.
\end{equation}
For each $f\in L^2(\Omega;\R^m)$, the above equation has a unique weak solution $u_\e \in H_0^1(\Omega;\R^m)$. Let $T_\e:L^2(\Omega;\R^m) \mapsto H^1_0(\Omega;\R^m)$ denote the linear map $f \mapsto u_\e$, i.e., $T_\e f = u_\e$. By the similar manner, denote by $T_0:L^2(\Omega;\R^m) \mapsto H^1_0(\Omega;\R^m)$ the linear map $f \to u_0$, where $u_0$ is the unique solution of the homogenized system
\begin{equation}\label{eq_L0u0}
\left\{
\begin{aligned}
\cL_0 u_0 &= f &\quad & \txt{in } \Omega, \\
u_0 &= 0 &\quad & \txt{on } \partial\Omega,
\end{aligned}
\right.
\end{equation}
where $\cL_0 = -\txt{div}(\widehat{A} \nabla )$ is the homogenized operator.

As before, let $\lambda_0 = \lambda_{0,L+j}$, $0\le j\le M-1$, be a Dirichlet eigenvalue of $\cL_0$ with multiplicity $M$ and $\lambda_{\e,L+j}$ be the eigenvalues of $\cL_\e$ converging to $\lambda_0$. Let $S_0$ be the spectral projection onto the eigenspace of $\cL_0$ corresponding to $\lambda_0$. In other words, for any $f\in L^2(\Omega;\R^m)$, define
\begin{equation}\label{eq.S0}
S_0 f = \Ag{f,\phi_{0,L+j}}\phi_{0,L+j}.
\end{equation}
Similarly, we denote by $S_\e$ the spectral projection onto the eigenspace of $\cL_\e$ corresponding to $\{\lambda_{\e,L+j}:0\le j\le M-1\}$, i.e.,
\begin{equation}\label{eq.Se}
S_\e f = \Ag{f,\phi_{\e,L+j}}\phi_{\e,L+j}.
\end{equation}
Let $\mathcal{R}(S_\e)$ and $\mathcal{R}(S_0)$ denote the ranges of $S_\e$ and $S_0$, respectively. 

Essentially, the asymptotic behavior of the eigenfunctions or eigenspaces is completely determined by those of $S_\e$ and $S_0$. The main result of this paper for Dirichlet problem is the following.
\begin{theorem}\label{thm_Ee}
	Let $A$ and $\Omega$ satisfy the same assumptions as Theorem \ref{thm_lambda}. Let $S_\e$ and $S_0$ be the spectral projections defined above. Then,
	\begin{equation}\label{est.Se-RS0}
	\norm{S_\e - S_0 - \e (\chi^\e \nabla + \Psi^{bl}) S_0} _{\mathcal{R}(S_0) \to L^2(\Omega)} \le C\e^{\frac{3}{2}-},
	\end{equation}
	and
	\begin{equation}\label{est.Se-L2}
	\norm{S_\e - S_0 - \e (\chi^\e \nabla + \Psi^{bl}) S_0 - \e S_0 (\chi^\e \nabla + \Psi^{bl})^*}_{L^2(\Omega) \to L^2(\Omega)} \le C \e^{\frac{3}{2}-},
	\end{equation}
	where $\chi^\e = \chi(\cdot/\e)$ is the first-order corrector, $\Psi^{bl}:\mathcal{R}(S_0) \mapsto \mathcal{R}(S_0)^\perp$ is a bounded linear operator given by
	\begin{equation}\label{eq_psi1bl}
	\Psi^{bl}g = \lambda_0^{-1} (\lambda_0^{-1}-T_0)^{-1} (I- S_0) K^{bl}g,
	\end{equation}
	and $C$ depends only on $\lambda_0, A$ and $\Omega$.
\end{theorem}

Observe that (\ref{est.Se-RS0}) is the expansion of $S_\e$ restricted in $\mathcal{R}(S_0)$ and (\ref{est.Se-L2}) is the expansion of $S_\e$ on the entire space $L^2(\Omega;\R^m)$. 
We should point out that the nontrivial operator $\Psi^{bl}$ introduced above plays a crucial role in correcting the first-order term involving the boundary layers. It is well-defined on $\mathcal{R}(S_0)$, since $(I- S_0) K^{bl}g \in \mathcal{R}(S_0)^\perp$ and the Fredholm theory implies $(\lambda_0^{-1} - T_0)^{-1}$ is a bounded linear operator on $\mathcal{R}(S_0)^\perp$. Actually, one can show that $\psi^{bl} = \Psi^{bl}g$ is the unique solution in $\mathcal{R}(S_0)^\perp$ of the following system: 
\begin{equation}
\left\{
\begin{aligned}
\cL_0 \psi^{bl} &= \bigg( -\bar{c}_{ijk} \frac{\partial^3}{\partial x_i \partial x_j \partial x_k} - \lambda_0 S_0   K^{bl} \bigg) g + \lambda_0 \psi^{bl} &\quad & \txt{in } \Omega, \\
\psi^{bl} &= K^{bl} g &\quad & \txt{on } \partial\Omega.
\end{aligned}
\right.
\end{equation}
Note that $\psi^{bl} + \mathcal{R}(S_0)$ is the solution set of the above system; see Remark \ref{rmk.psibl}.

In the case that $\lambda_0 = \lambda_{0,L}$ is a simple eigenvalue with eigenfunction $\phi_0 = \phi_{0,L}$, Theorem \ref{thm_Ee} implies that the eigenfunction $\phi_\e = \phi_{\e,L}$ satisfies
\begin{equation}\label{est.phie.simple}
\norm{\phi_\e - \phi_0 - \e \chi^\e \nabla \phi_0 - \e \Psi^{bl} \phi_0} _{L^2(\Omega)} \le C \e^{\frac{3}{2}-}.
\end{equation}
This particularly implies that
\begin{equation*}
\frac{\phi_\e - \phi_0}{\e} \longrightarrow \Psi^{bl} \phi_0 \quad \txt{weakly in } L^2(\Omega;\R^m),
\end{equation*}
and
\begin{equation*}
\frac{\phi_\e - \phi_0 - \e \chi^\e \nabla \phi_0}{\e} \longrightarrow \Psi^{bl} \phi_0 \quad \txt{strongly in } L^2(\Omega;\R^m),
\end{equation*}
which provide a positive answer to our previous question.

Our next result is the interior first-order expansion for the gradient of an eigenfunction corresponding to a simple eigenvalue.
\begin{theorem}\label{thm.dDphie}
	Let $A$ and $\Omega$ satisfy the same assumptions as Theorem \ref{thm_lambda}. Let $\lambda_0 = \lambda_{0,L}$ be a simple eigenvalue of $\cL_0$ and $\phi_0 = \phi_{0,L}, \phi_\e = \phi_{\e,L}$ be the eigenfunctions of $\cL_0$ and $\cL_\e$, correspondingly. Then, in the sense of $L^2(\Omega;\R^m)$, we have the following expansion of $\delta \nabla \phi_\e$,
	\begin{equation*}
		\begin{aligned}
			\delta \nabla \phi_\e 
			& =  \delta(I+\nabla \chi^\e )\nabla \phi_0 \\ 
			& \quad + \e\delta \big[ (\chi^\e I + \nabla \varUpsilon^\e) \nabla^2 \phi_0 + (I+ \nabla \chi^\e) \nabla \Psi^{bl}\phi_0 \big]  + O(\e^{\frac{3}{2}-}),
		\end{aligned}
	\end{equation*}
	where $\delta(x) = \txt{dist}(x,\partial\Omega)$, $\chi^\e$ and $\varUpsilon^\e$ are the first-order and second-order correctors, respectively, and the implicit constant depends only on $\lambda_0, A$ and $\Omega$.	
\end{theorem}

Note that the above theorem only provides the interior expansion as the distance function vanishes at the boundary. More precisely, it implies that for any $\Omega' \subset\subset \Omega$,
\begin{equation*}
	\norm{\nabla \phi_\e -(I+\nabla \chi^\e )\nabla \phi_0 -\e [ (\chi^\e I + \nabla \varUpsilon^\e) \nabla^2 \phi_0 + (I+ \nabla \chi^\e) \nabla \Psi^{bl}\phi_0 ] }_{L^2(\Omega')}\le C\e^{\frac{3}{2}-}.
\end{equation*}
where $C$ depends also on $\txt{dist}(\Omega',\partial\Omega)$. In particular, this implies that
\begin{equation*}
\frac{\nabla \phi_\e -(I+\nabla \chi^\e )\nabla \phi_0}{\e} \longrightarrow \nabla \Psi^{bl}\phi_0 \qquad \txt{weakly in } L^2_{\txt{loc}}(\Omega;\R^m).
\end{equation*}

\begin{remark}
	We emphasize that it is possible to specify how the constants in this paper depend on the eigenvalue $\lambda_0$. However, we will not try to track this dependence as it is difficult to verify the sharpness of the dependence on $\lambda_0$ (though this is possible with extra efforts).
\end{remark}

\begin{remark}
	We should point out that similar results as in Theorem \ref{thm_lambda}, \ref{thm_Ee} and \ref{thm.dDphie} hold for Neumann eigenvalue problem as well, due to the recent work in \cite{ShenZhuge16}, and the proofs are almost the same, which will be omitted. We will briefly introduce the problem and state the results in Theorem \ref{thm_Nlambda}, \ref{thm.NSe} and \ref{thm.dDphieN}.
\end{remark}

\begin{remark}
	As shown in \cite{MV97}, without any geometry condition on the domain $\Omega$, the first-order term in the expansions of the eigenvalues (or eigenfunctions) of $\cL_\e$ may not be unique and depend on the parameter $\e$. However, it is possible to generalize all of our results in this paper from the strictly convex domains to more general domains, such as domains of finite type \cite{Zhuge17}, at least for Dirichlet problem. This is because the geometry of domains only play a role in Theorem \ref{thm_L2NextOrder} and Theorem \ref{thm_NL2NextOrder}. By Remark \ref{rmk_ftype}, estimates of the same type can be extended, at least, to the Dirichlet problem in domains of finite type. As a result, Theorem \ref{thm_lambda}, \ref{thm_Ee} and \ref{thm.dDphie} may be generalized identically with a worse rate of $O(\e^{1+\frac{\alpha^*}{2} -})$, for some $\alpha^* \in (0,1]$.
\end{remark}

Finally, we give the organization of the paper and some key ideas in the proofs. In Section 2, we give preliminaries of the classical homogenization theory and recent results on the boundary layer problems. In Section 3, we prove Theorem \ref{thm_lambda}, following the same argument as in \cite{Prange13}. In particular, we observe a crucial first-order expansion for the operator $T_\e$
\begin{equation}\label{est.Tesim}
T_\e \sim T_0 + \e \chi^\e \nabla T_0 + \e K^{bl} T_0 +O(\e^{\frac{3}{2}-}).
\end{equation}
Theorem \ref{thm_Ee} is proved in Section 4, where the classical Riesz functional calculus will be our main tool, as $S_\e$ can be expressed by
\begin{equation*}
S_\e f = \frac{1}{2\pi i} \int_{\Gamma} (z-T_\e)^{-1} f dz,
\end{equation*}
where $\Gamma$ is a suitable contour in the complex plane $\C$. Then the expansion of $S_\e$ is reduced to that of $(z-T_\e)^{-1}$ for $z\in \Gamma$, which could be handled by the second resolvent identity and a more careful analysis combined with (\ref{est.Tesim}). Section 5 is devoted to the proof of Theorem \ref{thm.dDphie} which relies on both Theorem \ref{thm_lambda} and \ref{thm_Ee}. Under the assumption of Theorem \ref{thm.dDphie}, one can formally write
\begin{equation*}
\begin{aligned}
\cL_\e \phi_\e & = \lambda_\e \phi_\e \\
& \sim (\lambda_0 + \e \theta)(\phi_0 + \e \chi^\e \nabla \phi_0 + \e \Psi^{bl} \phi_0) + O(\e^{\frac{3}{2}-}).
\end{aligned}
\end{equation*}
Then, we exploit a new trick to deal with the interior gradient estimates in a circumstance with rough boundary data. In contrast to the classical argument in periodic homogenization in which $H^1$ error estimate is obtained priori to $L^2$ estimate, we reverse the argument in a sense that the interior ($\delta$-weighted) $H^1$ estimate follows from the $L^2$ estimate, regardless of the boundary condition, while the later is previously known by the homogenization of boundary layers. Finally, in Section 6, we state the parallel results for Neumann problem.

\section{Notations and Preliminaries}
Most notations in this paper are standard and the summation convention is used throughout for subscripts $i,j,k,\ell$ and supscripts $\alpha,\beta$, etc (it never applies to capital letters). For a 1-periodic function $f$, let $f^\e(x) = f(x/\e)$. We will also use the expression $a = b + O(r)$ to represent $|a-b| \le Cr$ (or $\norm{a-b}_{L^2(\Omega)} \le Cr$, depending the type of elements involved) for some constant $C$. As usual, the constant $C$ throughout this paper varies from line to line but never depends on $\e$. 

For any two Banach spaces $X$ and $Y$, let $\norm{\cdot}_{X\to Y}$ denote the operator norm from $X$ to $Y$. For a general function space $X$ $\subset L^1(\Omega;\R^m)$, let $\dot{X}$ denote the subspace of $X$ with elements satisfying the zero mean value property, i.e., $\dot{X} = \{ f\in X: \int_{\Omega} f = 0 \}$.

For each $1\le j\le d, 1\le \beta\le m$, let $\chi = (\chi_j^\beta) = (\chi_j^{1\beta},\chi_j^{2\beta},\cdots,\chi_j^{m\beta})$ denote the first-order correctors for $\cL_\e$, which are 1-periodic functions satisfying the cell problem
\begin{equation*}
	\left\{
	\begin{aligned}
		\cL_1 (\chi^{\beta}_{j} + P_j^\beta) &= 0 \qquad  \txt{ in } \T^d, \\
		\int_{\T^d} \chi_j^\beta & = 0,
	\end{aligned}
	\right.
\end{equation*}
where $P_j^\beta(x) = x_je^\beta $ with $e^\beta$ being $\beta$th Cartesian basis in $\R^m$. Recall that the homogenized matrix $\widehat{A} = (\hat{a}^{\alpha\beta}_{ij})$ is defined by
\begin{equation*}
\hat{a}^{\alpha\beta}_{ij} = \int_{\T^d} \bigg[ a^{\alpha\beta}_{ij} + a^{\alpha\gamma}_{ik} \frac{\partial}{\partial x_k} (\chi^{\gamma\beta}_j) \bigg] dx,
\end{equation*}
and the homogenized operator is given by $\cL_0 = -\txt{div}(\widehat{A} \nabla )$.

To study the first-order expansion of the eigenfunctions, we also need the second-order correctors for $\cL_\e$, $\varUpsilon = (\varUpsilon_{ij}^\beta) = (\varUpsilon_{ij}^{1\beta}, \varUpsilon_{ij}^{2\beta}, \cdots, \varUpsilon_{ij}^{m\beta})$, which are 1-periodic functions satisfying the cell problem
\begin{equation*}
	\left\{
	\begin{aligned}
		\cL_1 (\varUpsilon^{\beta}_{ij}) &= a_{ij}^{\cdot\beta} + a_{ik}^{\cdot\gamma} \frac{\partial}{\partial x_k} \chi_j^{\gamma\beta} - \hat{a}_{ij}^{\cdot\beta} + \frac{\partial}{\partial x_k}(a_{k i}^{\cdot\gamma} \chi_{j}^{\gamma\beta}) \qquad  \txt{ in } \T^d, \\
		\int_{\T^d} \varUpsilon_{ij}^\beta & = 0.
	\end{aligned}
	\right.
\end{equation*}
Under our standard assumptions on $A$, both $\chi$ and $\varUpsilon$ are smooth.

In the following context, we consider the first-order convergence rates for both Dirichlet and Neumann problems. First recall that the usual Dirichlet problem in a bounded domain $\Omega$ is
\begin{equation}\label{eq_Due}
\left\{
\begin{aligned}
\cL_\e u_\e(x) &= F(x) &\quad & \txt{in } \Omega, \\
u_\e(x) &= f(x) &\quad & \txt{on } \partial\Omega.
\end{aligned}
\right.
\end{equation}
For sufficiently regular $F$ and $f$, $u_\e$ converges to $u_0$ in $L^2(\Omega;\R^m)$ as $\e \to 0$, where $u_0$ is the solution of the homogenized equation with the same data,
\begin{equation}\label{eq_Du0}
\left\{
\begin{aligned}
\cL_0 u_0(x) &= F(x) &\quad & \txt{in } \Omega, \\
u_0(x) &= f(x) &\quad & \txt{on } \partial\Omega.
\end{aligned}
\right.
\end{equation}
Moreover, a formal asymptotic expansion for $u_\e$ is as follows
\begin{align}\label{eq_asymp}
	\begin{aligned}
		u_\e(x) = u_0(x) + & \e \bigg[ \chi_j\Big(\frac{x}{\e}\Big) \frac{\partial }{\partial x_j}u_0(x)  + v_{1,\e}^{bl}(x) \bigg] \\
		& + \e^2 \bigg[ \varUpsilon_{ij}\Big(\frac{x}{\e}\Big) \frac{\partial^2}{\partial x_i \partial x_j} u_0(x)  + v_{2,\e}^{bl}(x) \bigg] + \cdots,
	\end{aligned}
\end{align}
where $v_{n,\e}^{bl}$ is the $n$th-order boundary layer correction which solves a system with oscillating Dirichlet boundary data. In particular, $v_{1,\e}^{bl}$ is the solution of
\begin{equation}\label{eq_bl}
\left\{
\begin{aligned}
\cL_\e v_{1,\e}^{bl}(x) & = F^*(x) := -\bar{c}_{ijk} \frac{\partial^3 u_0(x)}{\partial x_i \partial x_j \partial x_k}  &\quad & \txt{in } \Omega, \\
v_{1,\e}^{bl}(x) &= -\chi_j\Big(\frac{x}{\e}\Big) \frac{\partial u_0}{\partial x_j}(x) &\quad & \txt{on } \partial\Omega,
\end{aligned}
\right.
\end{equation}
where $\bar{c}_{ijk}$ are constant.\footnote{In the scalar case that $m=1$, a careful computation shows that $F^* = 0$. But this fact will not be used in this paper.} The asymptotic analysis of $v_{1,\e}^{bl}$ in the above system is crucial for the higher-order expansion of $u_\e$. In the following theorem, we summarize the optimal results by far.
\begin{theorem}\label{thm_blL2}
	Let $A$ satisfy the ellipticity, periodicity and regularity assumptions and $\Omega$ be a bounded, smooth and strictly convex domain. Let $v_{1,\e}^{bl}$ be the solution of (\ref{eq_bl}). Then, there exists $v_{1,0}^{bl} \in L^2(\Omega;\R^m)$ independent of $\e$ such that
	\begin{equation*}
	\norm{v_{1,\e}^{bl} - v_{1,0}^{bl}}_{L^2(\Omega)} \le C\e^{\frac{1}{2}-} \norm{u_0}_{W^{3,\infty}(\Omega)}.
	\end{equation*}
	Moreover, $v_{1,0}^{bl}$ is the solution of
	\begin{equation}\label{eq.vblD}
	\left\{
	\begin{aligned}
	\cL_0 v_{1,0}^{bl} & = F^* &\quad & \txt{in } \Omega, \\
	v_{1,0}^{bl} &= f^* &\quad & \txt{on } \partial\Omega,
	\end{aligned}
	\right.
	\end{equation}
	where $f^* \in W^{1,p}(\partial\Omega;\R^m)$ for any $p\in (1,\infty)$.
\end{theorem}

The proof of Theorem \ref{thm_blL2} are contained in \cite{AKMP16,ShenZhuge16,ShenZhuge17}. In particular, the nearly sharp rate of convergence $O(\e^{\frac{1}{2}-})$ was obtained in \cite{AKMP16} for $d\ge 4$ and \cite{ShenZhuge16} for $d =3$ (The sharp rate $O(\e^{\frac{1}{4}-})$ for $d=2$ was also obtained in \cite{ShenZhuge16}, which may be used to improve Prange's result (\ref{est.Prange}) in 2-dimensional case). The explicit formula for the homogenized data $f^*$ was first discovered in \cite{AKMP16} and the $W^{1,p}$ regularity with arbitrary $p\in (1,\infty)$ has been proved in \cite{ShenZhuge17}.

As a corollary of Theorem \ref{thm_blL2}, we have
\begin{theorem}\label{thm_L2NextOrder}
	Under the same assumptions as Theorem \ref{thm_blL2}, the solutions of (\ref{eq_Due}) and (\ref{eq_Du0}) satisfy
	\begin{equation}\label{est_high_order}
	\norm{u_\e - u_0 - \e\chi_j^\e \frac{\partial u_0}{\partial x_j} - \e v_{1,0}^{bl} }_{L^2(\Omega)} \le  C\e^{\frac{3}{2}-}\norm{u_0}_{W^{3,\infty}(\Omega)},
	\end{equation}
	where $v_{1,0}^{bl}$ is given in Theorem \ref{thm_blL2}.
\end{theorem}

Next, we consider Neumann problem
\begin{equation}\label{eq_Nue}
\left\{
\begin{aligned}
\cL_\e u_\e(x) &= F(x) &\quad & \txt{in } \Omega, \\
\frac{\partial}{\partial \nu_\e} u_\e(x) &= g(x) &\quad & \txt{on } \partial\Omega,
\end{aligned}
\right.
\end{equation}
where $\frac{\partial}{\partial \nu_\e}  = n_i a_{ij}^\e \frac{\partial}{\partial x_j}$ is the conormal derivative and $n = (n_1,n_2,\cdots,n_d)$ is the normal vector. Similar to Dirichlet problem, the solution $u_\e$ of Neumann problem (\ref{eq_Nue}) converges to some function $u_0$ in $\dot{L}^2(\Omega;\R^m)$ as $\e \to 0$, and $u_0$ is the solution of the homogenized Neumann problem,
\begin{equation}\label{eq_Nu0}
\left\{
\begin{aligned}
\cL_0 u_0(x) &= F(x) &\quad & \txt{in } \Omega, \\
\frac{\partial}{\partial \nu_0} u_0(x) &= g(x) &\quad & \txt{on } \partial\Omega,
\end{aligned}
\right.
\end{equation}
where $\frac{\partial}{\partial \nu_0} = n_i \hat{a}_{ij}\frac{\partial}{\partial x_j}$. Moreover, we have a formal asymptotic expansion for the solution $u_\e$ of (\ref{eq_Nue}),
\begin{align}\label{eq_Nasymp}
	\begin{aligned}	
		u_\e(x)  = u_0(x) + &\e \bigg[ \chi_j\Big(\frac{x}{\e}\Big) \frac{\partial }{\partial x_j} u_0(x) + \tilde{v}_{1,\e}^{bl}(x) \bigg] \\
		+ &\e^2 \bigg[ \varUpsilon_{ij}\Big(\frac{x}{\e}\Big) \frac{\partial^2 }{\partial x_i x_j} u_0(x) + \tilde{v}_{2,\e}^{bl}(x) \bigg] + \cdots,
	\end{aligned}
\end{align}
where $\tilde{v}_{n,\e}^{bl}$ is the $n$th-order boundary layer correction corresponding to a problem with oscillating Neumann boundary data. Again, we are only interested in the first-order correction $\tilde{v}_{1,\e}^{bl} \in \dot{L}^2(\Omega;\R^m)$, which is given by
\begin{equation}\label{eq_Nbl}
\left\{
\begin{aligned}
\cL_\e \tilde{v}_{1,\e}^{bl} & = F^*= -\bar{c}_{ijk} \frac{\partial^3 u_0(x)}{\partial x_i \partial x_j \partial x_k} &\quad & \txt{in } \Omega, \\
\frac{\partial}{\partial \nu_\e } \tilde{v}_{1,\e}^{bl} &= \frac{1}{2}T_{ij}\cdot \nabla \Big( b_{ijk}^\e \frac{\partial u_0}{\partial x_k} \Big) + n_j \bar{c}_{ijk} \frac{\partial^2 u_0}{\partial x_i \partial x_k} &\quad & \txt{on } \partial\Omega.
\end{aligned}
\right.
\end{equation}
where $T_{ij} = n_i e_j - n_j e_i$ and $b_{ijk}$ are 1-periodic functions satisfying
\begin{equation}\label{eq.Bijk}
\frac{\partial}{\partial x_i} b_{ijk} = a_{jk} + a_{j\ell} \frac{\partial}{\partial x_\ell} \chi_k - \hat{a}_{jk} \quad \txt{and} \quad b_{ijk} = - b_{ijk}.
\end{equation}
Note that $T_{ij}, 1\le i,j\le d$, are tangential vector fields on the boundary and thence (\ref{eq_Nbl}) satisfies the condition of compatibility for Neumann problem. The derivation for the system of $\tilde{v}_{1,\e}^{bl}$ is contained, for example, in the proof of \cite[Theorem 9.1]{ShenZhuge16}. 

Analogous to Theorem \ref{thm_blL2}, we have
\begin{theorem}\label{thm_NblL2}
	Let $A$ and $\Omega$ satisfy the same assumptions as Theorem \ref{thm_blL2} and $\tilde{v}_{1,\e}^{bl}$ be the solution of (\ref{eq_Nbl}). Then, there exists $\tilde{v}_{1,0}^{bl} \in \dot{L}^2(\Omega;\R^m)$ independent of $\e$ such that
	\begin{equation*}
	\norm{\tilde{v}_{1,\e}^{bl} - \tilde{v}_{1,0}^{bl}}_{L^2(\Omega)} \le C\e^{1/2 -} \norm{u_0}_{W^{3,\infty}(\Omega)}.
	\end{equation*}
	Moreover, $\tilde{v}_{1,0}^{bl}$ is the solution of
	\begin{equation}
	\left\{
	\begin{aligned}
	\cL_0 \tilde{v}_{1,0}^{bl} & = F^* &\quad & \txt{in } \Omega, \\
	\frac{\partial}{\partial \nu_0} \tilde{v}_{1,0}^{bl} &= g^* &\quad & \txt{on } \partial\Omega,
	\end{aligned}
	\right.
	\end{equation}
	where $g^* \in W^{1,p}(\partial\Omega;\R^m)$ for any $p\in (1,\infty)$.
\end{theorem}
In the above theorem, the rate $O(\e^{\frac{3}{2}-})$ and the explicit formula of the homogenized data $g^*$ were obtained in \cite{ShenZhuge16}. The $W^{1,p}$ regularity of $g^*$ with any $p\in (1,\infty)$ was proved in \cite{ShenZhuge17}.

As a corollary, we have
\begin{theorem}[\cite{ShenZhuge16}, Thoerem 9.1]\label{thm_NL2NextOrder}
	Let $A$ and $\Omega$ satisfy the same assumptions as Theorem \ref{thm_blL2}. Then, the solutions of (\ref{eq_Nue}) and (\ref{eq_Nu0}) satisfy
	\begin{equation}\label{est_NHorder}
	\norm{u_\e - u_0 - \e\chi_j^\e \frac{\partial u_0}{\partial x_j} - \e \tilde{v}_{1,0}^{bl}}_{L^2(\Omega)} \le C\e^{\frac{3}{2}-}\norm{u_0}_{W^{3,\infty}(\Omega)},
	\end{equation}
	where $\tilde{v}_{1,0}^{bl}$ is given in Theorem \ref{thm_NblL2}.
\end{theorem}

Now observe that for arbitrary smooth $u_0$, the mappings $u_0 \mapsto  v_{1,0}^{bl}$ in Theorem \ref{thm_blL2} and $u_0 \mapsto \tilde{v}_{1,0}^{bl}$ in Theorem \ref{thm_NblL2} are both linear and bounded from $W^{3,\infty}(\Omega;\R^m)$ to $L^2(\Omega;\R^m)$. Thus, one can define bounded linear operators
\begin{equation}\label{def_K1}
K^{bl} u_0 = v_{1,0}^{bl},
\end{equation}
where $ v_{1,0}^{bl}$ is given by Theorem \ref{thm_blL2} for Dirichlet problem, and
\begin{equation}\label{def_tK1}
\widetilde{K}^{bl} u_0 = \tilde{v}_{1,0}^{bl},
\end{equation}
where $\tilde{v}_{1,0}^{bl}$ is given by Theorem \ref{thm_NblL2} for Neumann problem.
\begin{remark}\label{rmk_ftype}
	The results stated above are for strictly convex domains, while actually similar results are valid for domains of finite type, at least for Dirichlet problem, with the rate $O(\e^{\frac{3}{2}-})$ replaced by $O(\e^{1+\frac{1}{2}\alpha^*-})$ for some $\alpha^*\in (0,1]$ depending explicitly on $d$ and the type of $\Omega$; see \cite[Theorem 1.3]{Zhuge17}.
\end{remark}

\section{Expansion of Dirichlet Eigenvalues} 
This section is devoted to the proof of Theorem \ref{thm_lambda}. Let $A$ and $\Omega$ satisfy the assumptions of Theorem \ref{thm_lambda}. Recall the definitions of $T_\e$ and $T_0$ in (\ref{eq_Leue}) and (\ref{eq_L0u0}). Clearly, the family of operators $\{ T_\e: \e > 0 \}$ is uniformly bounded from $L^2(\Omega;\R^m)$ to $H_0^1(\Omega;\R^m)$ with respect to $\e$ and hence collectively compact on $L^2(\Omega;\R^m)$, i.e., the set $\{T_\e f: \norm{f}_{L^2(\Omega)} \le 1, \e > 0 \}$ is precompact in $L^2(\Omega;\R^m)$. Similarly, $T_0$ is bounded from $L^2(\Omega;\R^m)$ to $H_0^1 \cap H^2 (\Omega;\R^m)$ and therefore compact as an operator on $L^2(\Omega;\R^m)$. Moreover, by the classical homogenization theorem (see, e.g., \cite{ShenBook}), we have
\begin{equation}\label{est_L2rate}
\norm{T_\e - T_0}_{L^2(\Omega) \to L^2(\Omega)} \le C\e,
\end{equation}
where $C$ depends only on $A$ and $\Omega$.

In view of our setting, we note that the reciprocal of Dirichlet eigenvalues of $\cL_\e$, $\{\lambda_{\e,k}^{-1}: k\ge 1\}$, forms the sequence of all the eigenvalues of $T_\e$ in a decreasing order, with the same corresponding eigenfunctions, i.e., $T_\e \phi_{\e,k} = \lambda_{\e,k}^{-1} \phi_{\e,k}$ for every $k\ge 1$. Similarly, $\{\lambda_{0,k}^{-1}: k\ge 1\}$ is the sequence of eigenvalues of $T_0$ with the same corresponding eigenfunctions for each $k\ge 1$. For simplicity, throughout this paper, we define $\mu_{\e,k} = \lambda_{\e,k}^{-1}$ and $\mu_{0,k} =\lambda_{0,k}^{-1} $.

\begin{proof}[Proof of Theorem \ref{thm_lambda}]
	By the assumptions, $\mu_0 := \mu_{0,L} = \mu_{0,L+1} = \cdots = \mu_{0,L+M-1}$ is an eigenvalue of $T_0$ with multiplicity $M\ge 1$. Let $S_0 = S_0(\mu_0)$ be the spectral projection onto the the eigenspace of $T_0$ corresponding to $\mu_0$. By (\ref{cdn_lambda}), it is not hard to see that for sufficiently small $\e>0$, there exist exactly $M$ consecutive eigenvalues of $T_\e$, $\{ \mu_{\e,L+j}: j=0,1,\cdots, M-1\}$, such that $\mu_{\e,L+j}$ converges to $\mu_0$ for each $j$. Define
	\begin{equation}\label{eq_barmu}
	\bar{\mu}_\e = \frac{\mu_{\e,L} + \mu_{\e,L+1}+ \cdots + \mu_{\e,L+M-1}}{M}.
	\end{equation}
	It follows from Osborn's theorem in \cite[Theorem 3.1]{MV97} (or the proof of \cite[Theorem 3]{Os75}) and (\ref{est_L2rate}) that
	\begin{align}\label{est_mue_mu}
		\begin{aligned}
			&\Big|\bar{\mu}_\e -\mu_0 - \frac{1}{M} \Ag{(T_\e-T_0) \phi_{0,L+j},\phi_{0,L+j} } \Big| \\
			&  \le C \norm{T_\e-T_0}_{\mathcal{R}(S_0) \to L^2(\Omega)}^2 \\ 
			& \le  C\e^2,
		\end{aligned}
	\end{align}
	where the eigenfunctions $\{\phi_{0,L+j}: j =0,1,\cdots, M-1\}$ forms an (arbitrary) orthonormal basis of the eigenspace $\mathcal{R}(S_0)$. Note that the constant $C$ may depend on the eigenvalue $\mu_0$.
	
	To proceed, we need a lemma on the expansion of $T_\e$.
	\begin{lemma}\label{lem_TeT}
		Let $K^{bl}$ be the operator defined by (\ref{def_K1}). Then,
		\begin{equation*}
		\norm{T_\e - T_0 - \e \chi^\e_\ell \frac{\partial}{\partial x_\ell}  T_0 - \e K^{bl} T_0}_{\mathcal{R}(S_0) \to L^2(\Omega)} \le C\e^{\frac{3}{2}-}.
		\end{equation*}
	\end{lemma}
	Lemma \ref{lem_TeT} is a simple corollary of Theorem \ref{thm_L2NextOrder}. In fact, since $\Omega$ is smooth, the normalized eigenfunctions $\phi_{0,L+j} \in \mathcal{R}(S_0)$ are smooth and satisfy
	\begin{equation*}
	\norm{\nabla^k \phi_{0,L+j}}_{L^\infty(\Omega)} \le C_k,
	\end{equation*}
	for all $j=0,1,\cdots,M-1$, where $C$ depends on $\mu_0,k,A$ and $\Omega$. This implies that the identical embedding $\mathcal{R}(S_0) \subset W^{3,\infty}(\Omega;\R^m)$ is bounded. Now, fix $j$ and set $u_\e = T_\e \phi_{L+j}$ and $u_0 = T_0 \phi_{0,L+j}$. Then, Theorem \ref{thm_L2NextOrder} implies
	\begin{equation*}
	\norm{(T_\e - T_0 - \e \chi^\e_\ell \frac{\partial}{\partial x_\ell} T_0 - \e K^{bl} T_0) \phi_{0,L+j} }_{L^2(\Omega)} \le C\e^{\frac{3}{2}-}.
	\end{equation*}
	This proves the lemma.
	
	Next, we prove the following claim: there exists a fixed number $\gamma$ independent of $\e$ such that
	\begin{equation}\label{est.mue-mu0}
	| \bar{\mu}_\e - \mu_0 - \e \gamma | \le C \e^{\frac{3}{2}-},
	\end{equation}
	where $\gamma$ is given by
	\begin{equation}
	\gamma = \frac{\mu_0}{M} \Ag{K^{bl}\phi_{0,L+j}, \phi_{0,L+j}},
	\end{equation}
	and $K^{bl}$ is defined by (\ref{def_K1}).	
	
	To see this, first of all, it follows from (\ref{est_mue_mu}) and Lemma \ref{lem_TeT} that
	\begin{equation*}
	\big| \bar{\mu}_\e - \mu_0 - \frac{\e}{M}\Ag{\chi^\e_\ell \frac{\partial}{\partial x_\ell} T_0 \phi_{0,L+j} + K^{bl} T_0 \phi_{0,L+j}, \phi_{0,L+j}} \big| \le C\e^{\frac{3}{2}-}.
	\end{equation*}
	We then show that $|\Ag{\chi^\e_\ell \frac{\partial}{\partial x_\ell}  T_0 \phi_{0,L+j}, \phi_{0,L+j}}| \le C\e.$ Actually, since $\chi_\ell(y)$ is smooth, 1-periodic and of zero mean value, we can find a smooth function $B_\ell(y)$ so that $-\Delta B_\ell = \chi_\ell$. Thus, by the fact $T_0 \phi_{0,L+j} = \mu_0 \phi_{0,L+j}$, we have
	\begin{align*}
		|\Ag{\chi^\e_\ell \frac{\partial}{\partial x_\ell}  T_0 \phi_{0,L+j}, \phi_{0,L+j}}| & = \bigg| \int_{\Omega} \e^2 \Delta[B_\ell(x/\e)] \mu_0 \frac{\partial}{\partial x_\ell} \phi_{0,L+j}(x) \phi_{0,L+j}(x) dx \bigg| \\
		& = \bigg| \int_{\Omega} \e \nabla　B_\ell(x/\e) \cdot \nabla \bigg(\mu_0 \frac{\partial}{\partial x_\ell} \phi_{0,L+j}(x) \phi_{0,L+j}(x)\bigg) dx \bigg| \\
		& \le C \e \sum_{j=0}^{M-1} \norm{\phi_{0,L+j}}^2_{H^2(\Omega)} \le C\e,
	\end{align*}
	As a consequence, we obtain
	\begin{equation*}
	\big| \bar{\mu}_\e - \mu_0 - \frac{\e\mu_0}{M }\Ag{K^{bl}\phi_{0,L+j}, \phi_{0,L+j}} \big| \le C \e^{\frac{3}{2}-}.
	\end{equation*}
	It is important to note that $\Ag{K^{bl} \phi_{0,L+j}, \phi_{0,L+j}}$ is independent of the choice of the orthonormal basis $\{\phi_{0,L+j}: j=0,1,\cdots,M-1\}$. Therefore, we have proved the claim.

	Finally, we show that (\ref{est.mue-mu0}) implies (\ref{est.lambe}). Recall that $\mu_0 = \lambda_0^{-1}$ and $\mu_{\e,j} = \lambda_{\e,L+j}^{-1}$, and (\ref{cdn_lambda}) gives
	\begin{equation*}
	|\lambda_{\e,L+j} - \lambda_0| \le C\e, \qquad 0\le j\le M-1.
	\end{equation*}
	Observe that
	\begin{align*}
		\bar{\mu}_\e - \mu_0 & = \frac{\sum\limits_{j=0}^{M-1} (\lambda_0 - \lambda_{L+j}^\e) \prod\limits_{s\neq j} \lambda_{L+s}^\e }{ M \lambda_0 \prod\limits_{j=0}^{M-1} \lambda_{L+j-1}^\e } \\
		&= \frac{\sum\limits_{j=0}^{M-1} (\lambda_0 - \lambda_{L+j}^\e)}{M \lambda_0^2} + O(\e^2) \\
		& = \frac{\lambda_0 - \bar{\lambda}_\e}{\lambda_0^2} + O(\e^2),
	\end{align*}
	where $\bar{\lambda}_\e = M^{-1} \sum_{j=0}^{M-1} \lambda_{\e,L+j}$. This together with (\ref{est.mue-mu0}) implies
	\begin{equation*}
	\bar{\lambda}_\e - \lambda_0 = - \e \lambda^2_0 \gamma + O(\e^{\frac{3}{2}-}).
	\end{equation*}
	The theorem follows by letting $\theta = -\lambda_0^2 \gamma$.
\end{proof}

\section{Expansion of Dirichlet Eigenfunctions}
In this section, we will concentrate on the first-order expansion for spectral projections (or eigenspaces) of $T_\e$ or $\cL_\e$. Let $\mu_0$ be an eigenvalue of $T_0$ with multiplicity $M$ and $\mu_{\e,L+j}$ with $ 0\le j\le M-1$ be the eigenvalues of $T_\e$ that converge to $\mu_0$ as $\e \to 0$. Let $\{\phi_{0,L+j}: j=0,1,\cdots, M-1\}$ be an orthonormal basis of the eigenspace of $T_0$ corresponding to $\mu_0$ and let $\{ \phi_{\e,L+j}: j=0,1,\cdots,M-1 \}$ be the orthonormal eigenfunctions of $T_\e$ corresponding to $\{\mu_{\e,L+j}: j=0,1,\cdots,M-1\}$.

Suppose $\Gamma \subset \C$ is a circle centered at $\mu_0$ with fixed radius such that the only eigenvalue of $T_0$ enclosed by $\Gamma$ is $\mu_0$ and the only eigenvalues of $T_\e$ enclosed are exactly $\{ \mu_{\e,L+j}: 0\le j\le M-1\}$. It is harmless and crucial to assume that the distance form $\mu_{\e,L+j}$ to $\Gamma$ is uniformly bounded below. This implies that the resolvents $(z-T_0)^{-1}$ and $(z-T_\e)^{-1}$ are bounded on $L^2(\Omega;\R^m)$ uniformly for any $z\in \Gamma$ and sufficiently small $\e$. By the theory of Riesz functional calculus, the spectral projections defined in (\ref{eq.S0}) and (\ref{eq.Se}) can be expressed by
\begin{equation*}
S_0f = \frac{1}{2\pi i} \int_{\Gamma} (z-T_0)^{-1} f dz,
\end{equation*}
and
\begin{equation*}
S_\e f = \frac{1}{2\pi i} \int_{\Gamma} (z-T_\e)^{-1} f dz.
\end{equation*}
Some basic properties about the projections are listed below.
\begin{proposition}\label{prop_E}
	Let $S_\e,S_0$ be defined as above. Then,
	
	(i) $S_0^2 = S_0 = S_0^*$ and $S_\e^2 = S_\e = S_\e^*$;
	
	
	(ii) The projection $S_0$ is bounded from $L^2(\Omega;\R^m)$ to Sobolev space $H^k(\Omega;\R^m)$ for any $k\ge 0$;
	
	(iii) The projection $S_0$ can be extended naturally to a bounded linear operator on $H^{-s}(\Omega;\R^m)$ such that $S_0$ is bounded from $H^{-s}(\Omega;\R^m)$ to $H^k(\Omega;\R^m)$ for arbitrary $s,k \ge 0$.
\end{proposition}

\begin{proof}
	Part (ii) follows from the fact that $S_0 f = \Ag{f,\phi_{0, L+j}}\phi_{0, L+j}$ and $\phi_{0, L+j}$ are smooth. To see part (iii), note that $\Ag{f,\phi_{0, L+j}}$ can be naturally extended to the pair action $\Ag{f,\phi_{0, L+j}}_{H^{-s}(\Omega) \times H^s(\Omega)}$ for any $s\ge 0$. Thus
	\begin{equation*}
	S_0 f = \Ag{f,\phi_{0, L+j}}_{H^{-s}(\Omega) \times H^s(\Omega)} \phi_{0, L+j},
	\end{equation*}
	defines a bounded linear operator from $H^{-s}(\Omega;\R^m)$ to $H^k(\Omega;\R^m)$.
\end{proof}

Then, we have the following zero-order expansion for the spectral projection $S_\e$ that will be used later.
\begin{lemma}\label{lem_EeEL2}
	For sufficiently small $\e > 0$, it holds
	\begin{equation*}
	\norm{S_\e - S_0}_{L^2(\Omega) \to L^2(\Omega)}  \le C\e,
	\end{equation*}
	where $C$ depends only on $\mu_0,A$ and $\Omega$.
\end{lemma}
\begin{proof}
	For any $g\in L^2(\Omega;\R^m)$,
	\begin{align}\label{eq_Eeg}
		\begin{aligned}
			S_\e g - S_0 g & = \frac{1}{2\pi i} \int_{\Gamma} ((z-T_\e)^{-1} - (z-T_0)^{-1}) g dz \\
			& = \frac{1}{2\pi i} \int_{\Gamma} (z-T_\e)^{-1}(T_\e - T_0) (z-T_0)^{-1} g dz,
		\end{aligned}
	\end{align}
	where we have used the second resolvent identity
	\begin{equation}\label{eq_ID}
	(z-T_\e)^{-1} - (z-T_0)^{-1} = (z-T_\e)^{-1}(T_\e - T_0) (z-T_0)^{-1}.
	\end{equation}
	Recall that $\norm{T_\e - T_0}_{L^2(\Omega)\to L^2(\Omega)} \le C\e$, and $ (z-T_0)^{-1}$ and $(z-T_\e)^{-1}$ are uniformly bounded in $L^2(\Omega;\R^m)$, provided $\e$ is sufficiently small. The lemma follows from (\ref{eq_Eeg}) easily.
\end{proof}

Since the identity (\ref{eq_Eeg}) is not sufficient to study the first-order expansion, we apply the identity (\ref{eq_ID}) in (\ref{eq_Eeg}) again and obtain
\begin{align}\label{eq_Eephi_g}
	\begin{aligned}
		S_\e g - S_0 g & = \frac{1}{2\pi i} \int_{\Gamma} (z-T_0)^{-1}(T_\e - T_0) (z-T_0)^{-1} g dz \\
		& \quad + \frac{1}{2\pi i} \int_{\Gamma} (z-T_\e)^{-1}(T_\e - T_0) (z-T_0)^{-1} (T_\e - T_0) (z-T_0)^{-1} g dz
\end{aligned}\end{align}
For the same reason as in the proof of Lemma \ref{lem_EeEL2}, the second term of (\ref{eq_Eephi_g}) is bounded by $C\e^2 \norm{g}_{L^2(\Omega)}$, which is a higher-order error. Therefore, it suffices to consider the first-order expansion of
\begin{equation*}
\frac{1}{2\pi i} \int_{\Gamma} (z-T_0)^{-1}(T_\e - T_0) (z-T_0)^{-1} g dz.
\end{equation*}

To this end, the lemmas below are crucial for us.
\begin{lemma}\label{lem_zT_f}
	For any $f\in L^2(\Omega;\R^m)$, we have
	\begin{equation}\label{eq_zt-1}
	(z-T_0)^{-1} f = z^{-1} f + z^{-1}(z-T_0)^{-1} T_0 f.
	\end{equation}
	In particular, if $g \in \mathcal{R}(S_0)$, then
	\begin{equation}\label{eq.zT0}
	(z-T_0)^{-1} g = \frac{g}{z-\mu_0}.
	\end{equation} 
\end{lemma}
\begin{proof}
	It is easy to see (\ref{eq_zt-1}) by acting $z-T_0$ on both sides. And (\ref{eq.zT0}) follows from (\ref{eq_zt-1}) and the fact $T_0 g = \mu_0 g$.
\end{proof}

\begin{lemma}\label{lem.T0feg}
	Let $f\in C^\sigma(\T^d;\R^m)$ for some $\sigma\in (0,1)$,  $\int_{\T^d} f = 0$ and $g \in H^1(\Omega;\R^m)$. Then,
	\begin{equation*}
	\norm{T_0(f^\e g)}_{L^2(\Omega)} \le C \e \norm{f}_{C^\sigma(\T^d)} \norm{g}_{H^1(\Omega)},
	\end{equation*}
	where $f^\e(x) = f(x/\e)$.
\end{lemma}
\begin{proof}
	Let $v_\e = T_0(f^\e g)$, i.e., $v_\e$ satisfies $\cL_0 v_\e = f^\e g$ in $\Omega$ and $v_\e = 0$ on $\partial\Omega$. Then integrating the equation against $v_\e$ and using integration by parts, one has
	\begin{equation*}
	\int_{\Omega} \widehat{A}\nabla v_\e \cdot \nabla v_\e = \int_{\Omega} f(x/\e) g(x) v_\e(x)dx.
	\end{equation*}
	Since the mean value of $f$ is zero, there is a unique periodic function $F$ with mean value zero such that $-\Delta F = f$ and $\norm{\nabla F}_{L^\infty(\T^d)} \le C \norm{f}_{C^\sigma(\T^d)}$. Then by the ellipticity condition and the Poincar\'{e} inequality,
	\begin{align*}
		\Lambda^{-1} \norm{\nabla v_\e}_{L^2(\Omega)}^2 &\le 
		\int_{\Omega} \widehat{A} \nabla v_\e \cdot \nabla v_\e \\
		& = \e \int_{\Omega} \nabla F(x/\e) \cdot \nabla( g(x) v_\e(x)) dx \\
		&\le C \e \norm{f}_{C^\sigma(\T^d)} \norm{\nabla v_\e}_{L^2(\Omega)} \norm{g}_{H^1(\Omega)}.
	\end{align*}
	This implies the desired estimate.
\end{proof}

\begin{proof}[Proof of Theorem \ref{thm_Ee}, part I]
	Fix some $g \in \mathcal{R}(S_0)$ with $\norm{g}_{L^2(\Omega)} =1$. As mentioned before, (\ref{eq_Eephi_g}) and (\ref{eq.zT0}) imply
	\begin{equation*}
	S_\e g - S_0 g = \frac{1}{2\pi i} \int_{\Gamma} (z-\mu_0)^{-1} (z-T_0)^{-1}(T_\e - T_0) g dz + O(\e^2),
	\end{equation*}
	in the sense of $L^2(\Omega;\R^m)$. By Lemma \ref{lem_TeT} and $T_0g = \mu_0 g$, we have
	\begin{align}\label{est.Se-S0}
		\begin{aligned}	
			S_\e g - S_0 g = & \frac{1}{2\pi i} \int_{\Gamma}  \mu_0(z-\mu_0)^{-1} (z-T_0)^{-1} (\e\chi^\e_j \frac{\partial}{\partial x_j} g )dz \\
			& \qquad + \frac{1}{2\pi i} \int_{\Gamma} \mu_0(z-\mu_0)^{-1}(z-T_0)^{-1} \e K^{bl} g dz + O(\e^{\frac{3}{2}-}).
		\end{aligned}
	\end{align}
	We first deal with the first term on the right-hand side of (\ref{est.Se-S0}). Observe that Lemma \ref{lem_zT_f} implies
	\begin{equation}\label{eq.z-T0}
	(z-T_0)^{-1} (\e\chi^\e_j \frac{\partial}{\partial x_j} g) = z^{-1} \e\chi^\e_j \frac{\partial}{\partial x_j} g + z^{-1}(z-T_0)^{-1}T_0(\e\chi^\e_j \frac{\partial}{\partial x_j} g).
	\end{equation}
	Note that Lemma \ref{lem.T0feg} implies $T_0(\e\chi^\e \nabla g) \le C\e^2$. Combining this with (\ref{est.Se-S0}) and (\ref{eq.z-T0}), we have
	\begin{align}
		\begin{aligned}	\label{est.Seg-S0g}
			S_\e g - S_0 g & = \frac{1}{2\pi i} \int_{\Gamma} \mu_0(z-\mu_0)^{-1} z^{-1}dz \cdot \e\chi^\e_j \frac{\partial}{\partial x_j} g + \e \Psi^{bl}g + O(\e^{\frac{3}{2}-}) \\
			& = \e\chi^\e_j \frac{\partial}{\partial x_j} g + \e \Psi^{bl}g + O(\e^{\frac{3}{2}-}),
		\end{aligned}
	\end{align}
	where $\Psi^{bl} g$ is defined by
	\begin{equation}\label{eq_defpsi}
	\Psi^{bl}g = \frac{1}{2\pi i} \int_{\Gamma} \mu_0(z-\mu_0)^{-1}(z-T_0)^{-1} K^{bl} g dz,
	\end{equation}
	and we have used the fact
	\begin{equation*}
	\frac{1}{2\pi i} \int_{\Gamma} \mu_0(z-\mu_0)^{-1} z^{-1}dz = 1.
	\end{equation*}
	Note that $S_0 g = g$ for $g\in \mathcal{R}(S_0)$. Thus, (\ref{est.Seg-S0g}) implies the desired estimate (\ref{est.Se-RS0}).
	
	Finally, we need to show that the operator $\Psi^{bl}$ above satisfies (\ref{eq_psi1bl}). A computation shows that
	\begin{align*}
		\Psi^{bl}g  &= \frac{1}{2\pi i} \int_{\Gamma} \frac{\mu_0}{z-\mu_0}(z-T_0)^{-1} K^{bl} g dz \\
		& = - \frac{1}{2\pi i}\int_{\Gamma} (z-T_0)^{-1} K^{bl} g dz + \frac{1}{2\pi i} \int_{\Gamma} \frac{z}{z-\mu_0}(z-T_0)^{-1} K^{bl} g dz \\
		& = - S_0 K^{bl}g + \frac{1}{2\pi i} \int_{\Gamma} \frac{z}{z-\mu_0}(z-T_0)^{-1} K^{bl} \phi dz.
	\end{align*}
	Using Lemma \ref{lem_zT_f}, we have
	\begin{align*}
		& \frac{1}{2\pi i} \int_{\Gamma} \frac{z}{z-\mu_0}(z-T_0)^{-1} K^{bl} g dz \\
		&\qquad = \frac{1}{2\pi i} \int_{\Gamma} \frac{z}{z-\mu_0} z^{-1} K^{bl} g dz + \frac{1}{2\pi i} \int_{\Gamma}  \frac{z}{z-\mu_0}z^{-1}(z-T_0)^{-1} T_0 K^{bl} g dz \\
		& \qquad = K^{bl} g + \frac{1}{2\pi i} \int_{\Gamma} \frac{1}{z-\mu_0}(z-T_0)^{-1} T_0 K^{bl} g dz.
	\end{align*}
	It follows that
	\begin{equation}\label{eq.Psig}
	\Psi^{bl}g = (I- S_0) K^{bl}g + \frac{1}{2\pi i} \int_{\Gamma} \frac{1}{z-\mu_0}(z-T_0)^{-1} T_0 K^{bl} g dz.
	\end{equation}
	
	To proceed, we now claim that $T_0$ commutes with $(z-T_0)^{-1}$. Actually, by replacing $f$ with $(z-T_0)^{-1} f$ in (\ref{eq_zt-1}) and then acting $z-T_0$ on both sides, we obtain
	\begin{equation}
	(z-T_0)^{-1} f = z^{-1} f + z^{-1}T_0 (z-T_0)^{-1} f.
	\end{equation}
	This implies the claim in view of (\ref{eq_zt-1}). Hence, the last term on the right-hand side of (\ref{eq.Psig}) is equal to $\mu_0^{-1} T_0(\Psi^{bl}g)$. As a result,
	\begin{equation}\label{eq_psiKK}
	\Psi^{bl}g = (I- S_0) K^{bl}g + \mu_0^{-1} T_0(\Psi^{bl}g),
	\end{equation}
	or equivalently,
	\begin{equation*}
		(\mu_0-T_0)(\Psi^{bl}g) = \mu_0(I- S_0) K^{bl}g.
	\end{equation*}
	Observe that $(I- S_0) K^{bl}g$ is orthogonal to $\mathcal{R}(S_0)$, i.e., $(I- S_0) K^{bl}g \in \mathcal{R}(S_0)^{\perp}$. Because $T_0$ is compact on $L^2(\Omega;\R^m)$, the Fredholm theorem allows us to invert $\mu_0-T_0$ on $\mathcal{R}(S_0)^{\perp}$. Consequently, we obtain
	\begin{equation*}
		\Psi^{bl}g = \mu_0 (\mu_0-T_0)^{-1}(I- S_0) K^{bl}g.
	\end{equation*}
	Note that the operator $\Psi^{bl}: \mathcal{R}(S_0) \mapsto \mathcal{R}(S_0)^{\perp}$ is linear and bounded. This proves (\ref{eq_psi1bl}).
\end{proof}

\begin{remark}\label{rmk.psibl}
	The abstract formula (\ref{eq_psi1bl}) can be interpreted as a certain system. To see this, letting $\psi^{bl} = \Psi^{bl}g$ and applying $\cL_0$ to (\ref{eq_psiKK}), we obtain
	\begin{align*}
		\cL_0 \psi^{bl} = \bigg( -\bar{c}_{ijk} \frac{\partial^3}{\partial x_i \partial x_j \partial x_k} - \lambda_0 S_0 K^{bl} \bigg) g + \lambda_0 \psi^{bl},
	\end{align*}
	where we have used (\ref{eq.vblD}) and the fact $\cL_0 S_0 = \lambda_0 S_0$. To find out the boundary condition, note that $S_0 K^{bl} g$ and $T_0(\Psi^{bl} g)$ are both vanishing on $\partial\Omega$. Thus (\ref{eq_psiKK}) implies that $\psi^{bl}|_{\partial\Omega} = K^{bl}g|_{\partial\Omega}$.  Consequently, we obtain the equation for $\psi^{bl}$: 
	\begin{equation}
	\left\{
	\begin{aligned}
	\cL_0 \psi^{bl} &= \bigg( -\bar{c}_{ijk} \frac{\partial^3}{\partial x_i \partial x_j \partial x_k} - \lambda_0 S_0 K^{bl} \bigg) g + \lambda_0 \psi^{bl} &\quad & \txt{in } \Omega, \\
	\psi^{bl} &= K^{bl} g &\quad & \txt{on } \partial\Omega.
	\end{aligned}
	\right.
	\end{equation}
	Note that $\psi^{bl} + \mathcal{R}(S_0)$ forms the set of all solutions for the above system.
\end{remark}

\begin{corollary}\label{coro_SimEigen}
	Let $A$ and $\Omega$ satisfy the same assumptions as Theorem \ref{thm_lambda}. Let $\lambda_0 = \lambda_{0,L}$ be a simple Dirichlet eigenvalue of $\cL_0$ with eigenfunction $\phi_0 = \phi_{0,L}$ and $\lambda_\e = \lambda_{\e,L}$ be the Dirichlet eigenvalue of $\cL_\e$ with eigenfunction $\phi_\e = \phi_{\e,L}$. Then for $\e>0$ sufficiently small,
	\begin{equation}\label{est.phie.cor}
	\norm{\phi_\e - \phi_0 - \e \chi^\e \nabla \phi_0 - \e \Psi^{bl} \phi_0} _{L^2(\Omega)} \le C \e^{\frac{3}{2}-},
	\end{equation}
	where the operator $\Psi^{bl}$ is given by (\ref{eq_psi1bl})
\end{corollary}
\begin{proof}
	Since $\lambda_0$ is simple, $\mathcal{R}(S_0) = \txt{Span}\{ \phi_0 \}$. In view of (\ref{cdn_lambda}), $\lambda_\e$ is also simple for sufficiently small $\e$ and $\mathcal{R}(S_\e) = \txt{Span}\{ \phi_\e \}$. Then, (\ref{est.Se-RS0}) implies
	\begin{equation}
	\norm{ S_\e \phi_0 - \phi_0 - \e \chi^\e \nabla \phi_0 - \e \Psi^{bl} \phi_0} _{L^2(\Omega)} \le C \e^{\frac{3}{2}-}.
	\end{equation}
	Then it is sufficient to show $\norm{S_\e \phi_0 - \phi_\e}_{L^2(\Omega)} \le C\e^2$. Actually, by $S_\e \phi_0 = \Ag{\phi_0, \phi_\e} \phi_\e$ and the fact $\Ag{\phi_\e - \phi_0, \phi_\e + \phi_0} = 0$, we have
	\begin{equation}
	S_\e \phi_0 - \phi_\e = \Ag{\phi_0, \phi_\e - \phi_0} \phi_\e = \frac{1}{2}\Ag{\phi_0 - \phi_\e, \phi_\e - \phi_0} \phi^\e.
	\end{equation}
	The desired estimate follows from $\norm{\phi_\e - \phi_0}_{L^2(\Omega)} \le C\e$, which is a corollary of Lemma \ref{lem_EeEL2}.
\end{proof}


Now we are in a position to investigate $S_\e$ as an operator defined on $\mathcal{R}(S_0)^{\perp}$.
\begin{lemma}\label{lem_Eeperp}
	For sufficiently small $\e > 0$, it holds
	\begin{equation}
	\norm{S_\e - \e S_0 (\chi^\e \nabla +  \Psi^{bl})^* }_{\mathcal{R}(S_0)^\perp \to L^2(\Omega)} \le C\e^{\frac{3}{2}-}.
	\end{equation}
\end{lemma}
Note that in Lemma \ref{lem_Eeperp}, $\Psi^{bl*}$ is the adjoint operator of $\Psi^{bl}$, and by definition, $\Psi^{bl}$ is a bounded linear operator from $W^{3,\infty}(\Omega;\R^m)$ to $L^2(\Omega;\R^m)$. Thus, $\Psi^{bl*}$ is a bounded operator from $L^2(\Omega;\R^m)$ to $W^{3,\infty}(\Omega;\R^m)' \subset H^{-s}(\Omega;\R^m)$ for some $s>0$. Then, by Proposition \ref{prop_E} part (iii), $S_0  \Psi^{bl*}$ is a well-defined bounded operator on $L^2(\Omega;\R^m)$.

\begin{proof}[Proof of Lemma \ref{lem_Eeperp}]
	Let $g\in \mathcal{R}(S_0)^\perp$ and $\norm{g}_{L^2(\Omega)} = 1$ . For a given $h\in L^2(\Omega;\R^m)$ with $\norm{h}_{L^2(\Omega)} = 1$, let $h = h_1 + h_2$ with $h_1\in \mathcal{R}(S_0), h_2 \in \mathcal{R}(S_0)^\perp$. Consider
	\begin{equation*}
	\Ag{h,S_\e g} = \Ag{h_1,S_\e g} + \Ag{h_2,S_\e g}.
	\end{equation*}
	For the first term, using $\Ag{h_1,g} = 0, S_0 h = h_1$ and (\ref{est.Se-RS0}), we have
	\begin{align}
		\begin{aligned}
			\Ag{h_1,S_\e g} & = \Ag{S_\e h_1, g} \\
			& = \Ag{ h_1 + \e(\chi^\e \nabla + \Psi^{bl}) S_0 h_1, g} + O(\e^{\frac{3}{2}-}) \\
			& =  \Ag{  h, \e S_0 (\chi^\e \nabla + \Psi^{bl})^* g} + O(\e^{\frac{3}{2}-}).
		\end{aligned}
	\end{align}
	On the other hand, using $S_0 h_2 = 0, S_\e = S_\e^2$ and Lemma \ref{lem_EeEL2}, we obtain
	\begin{align}
		\Ag{h_2,S_\e g} = \Ag{(S_\e-S_0)^2 h_2, g} = O(\e^2).
	\end{align}
	Therefore, we have
	\begin{equation*}
	\Ag{h,S_\e g} = \Ag{  h, \e S_0 (\chi^\e \nabla + \Psi^{bl})^* g} + O(\e^{\frac{3}{2}-}).
	\end{equation*}
	This implies the desired estimate.
\end{proof}

\begin{lemma}\label{lem.Psi*}
	For sufficiently small $\e>0$, it holds
	\begin{equation}\label{est.Psi*}
	\norm{S_0 (\chi^\e \nabla +  \Psi^{bl})^* }_{\mathcal{R}(S_0) \to L^2(\Omega)} \le C\e.
	\end{equation}
\end{lemma}
\begin{proof}
	First, we show 
	\begin{equation}\label{est.S0chi}
	\norm{S_0 (\chi^\e \nabla)^*}_{\mathcal{R}(S_0) \to L^2(\Omega)} \le C\e.
	\end{equation}
	Actually, let $g\in \mathcal{R}(S_0)$ and $h\in L^2(\Omega;\R^m)$. Then, by a similar argument as Lemma \ref{lem.T0feg}, we have
	\begin{align*}
		|\Ag{h,S_0 (\chi^\e \nabla)^* g}| & = \bigg| \int_{\Omega} \chi(x/\e) \nabla(S_0 h)(x) \cdot g(x) dx \bigg| \\
		& \le C\e \norm{S_0 h}_{H^2(\Omega)} \norm{g}_{H^1(\Omega)} \\
		& \le C \e \norm{h}_{L^2(\Omega)} \norm{g}_{L^2(\Omega)},
	\end{align*}
	where we also used Proposition \ref{prop_E} in the last inequality. This implies (\ref{est.S0chi}) as desired.
	
	Next, we show that $S_0   \Psi^{bl*}   S_0 = 0$ on $L^2(\Omega;\R^m)$. Actually, recalling that $\Psi^{bl}$ maps $\mathcal{R}(S_0)$ to $\mathcal{R}(S_0)^\perp$, we have $S_0   \Psi^{bl}   S_0 = 0$ on $L^2(\Omega;\R^m)$, which implies $S_0   \Psi^{bl*}   S_0 = 0$ on $L^2(\Omega;\R^m)$. This, together with (\ref{est.S0chi}), proves the lemma.
\end{proof}

\begin{proof}[Proof of Theorem \ref{thm_Ee}, part II]
	For any function $g \in L^2(\Omega;\R^m)$ write $g = g_1 + g_2$ such that $g_1 \in \mathcal{R}(S_0)$ and $g_2\in \mathcal{R}(S_0)^\perp$. Note that $S_0 g_2 = 0$. It follows that
	\begin{align}\label{est_EegEg}
		\begin{aligned}
			&\norm{S_\e g - S_0 g - \e (\chi^\e_j \partial_j + \Psi^{bl})   S_0 g- \e S_0   (\chi^\e_j \partial_j + \Psi^{bl})^* g}_{L^2(\Omega)} \\
			&\quad \le \norm{S_\e g_1 - S_0 g_1 - \e (\chi^\e_j \partial_j + \Psi^{bl})   S_0 g_1}_{L^2(\Omega)} \\
			& \qquad + \e \norm{ S_0   (\chi^\e_j \partial_j + \Psi^{bl})^* g_1}_{L^2(\Omega)} \\
			& \qquad+ \norm{S_\e g_2 - \e S_0   (\chi^\e_j \partial_j + \Psi^{bl})^* g_2}_{L^2(\Omega)}.
		\end{aligned}
	\end{align}
	All the three term are bounded by $C\e^{\frac{3}{2}-}\norm{g}_{L^2(\Omega)}$ thanks to (\ref{est.Se-RS0}), Lemma \ref{lem.Psi*} and Lemma \ref{lem_Eeperp}, respectively. The proof is complete.
\end{proof}


\section{Interior expansion of gradient}
This section is devoted to the proof of Theorem \ref{thm.dDphie}. Throughout, we assume $\lambda_0 = \lambda_{0,L}$ is a simple Dirichlet eigenvalue for some $L\ge 1$ and let $\phi_0 = \phi_{0,L}, \phi_\e = \phi_{\e,L}$ and $\lambda_\e = \lambda_{\e,L}$. By Theorem \ref{thm_lambda} and Corollary \ref{coro_SimEigen}, we have
\begin{equation*}
\begin{aligned}
\cL_\e \phi_\e & = \lambda_\e \phi_\e \\
&= \lambda_0 \phi_0 + \e(\theta \phi_0 + \lambda_0 \Psi^{bl}\phi_0 ) + \e \lambda_0 \chi^\e \nabla \phi_0 + R_\e,
\end{aligned}
\end{equation*}
where $\norm{R_\e}_{L^2(\Omega)} \le C\e^{\frac{3}{2}-}$. Thus, one can write
\begin{equation*}
\begin{aligned}
\phi_\e & = \lambda_0 T_\e \phi_0 + \e T_\e (\theta \phi_0+ \lambda_0 \Psi^{bl}\phi_0 )\\
& \qquad   + \e T_\e(\chi^\e \nabla \phi_0) + T_\e (R_\e).
\end{aligned}
\end{equation*}
By the energy estimate and Lemma \ref{lem.T0feg} (the lemma holds for $T_\e$ as well), we have
\begin{equation*}
\norm{\nabla T_\e (R_\e)}_{L^2(\Omega)} \le C\e^{\frac{3}{2}-},
\end{equation*}
and
\begin{equation*}
\norm{\nabla T_\e (\chi^\e \nabla \phi_0)}_{L^2(\Omega)} \le C\e.
\end{equation*}
Consequently,
\begin{equation}\label{eq.Dphie}
\begin{aligned}
\nabla \phi_\e & = \lambda_0 \nabla T_\e \phi_0 \\
& \qquad + \e \nabla T_\e (\theta \phi_0 + \lambda_0 \Psi^{bl}\phi_0 ) + O(\e^{\frac{3}{2}-}), \quad \txt{in } L^2(\Omega;\R^m).
\end{aligned}
\end{equation}
Hence, it is sufficient to study the asymptotic behavior for $\nabla T_\e f$ with some $f$ independent of $\e$. Note that $u_\e = T_\e f$ is the weak solution of (\ref{eq_Leue}). The following theorem should be well-known.
\begin{theorem}\label{thm.ue2nd}
	Let $u_\e = T_\e f$ and $u_0 = T_0 f$ for properly smooth $f$. Then
	\begin{equation}\label{est.ueH2}
	\norm{ u_\e - u_0 - \e \chi^\e \nabla u_0  - \e v_{1,\e}^{bl} - \e^2 \varUpsilon^\e \nabla^2 u_0 - \e^2 v_{2,\e}^{bl} }_{H^1(\Omega)} \le C\e^2 \norm{u_0}_{H^3(\Omega)}.
	\end{equation}
	where $v_{1,\e}^{bl}$ is given by (\ref{eq_bl}) and $v_{2,\e}^{bl}$ is given by
	\begin{equation}\label{eq_bl2}
	\left\{
	\begin{aligned}
	\cL_\e v_{2,\e}^{bl}(x) & = 0  &\quad & \txt{in } \Omega, \\
	v_{2,\e}^{bl}(x) &= -\varUpsilon^\e \nabla^2 u_0 &\quad & \txt{on } \partial\Omega.
	\end{aligned}
	\right.
	\end{equation}
\end{theorem}
\begin{proof}[Sketch of the proof]
	Define
	\begin{equation*}
	w_\e = u_\e - u_0 - \e \chi^\e \nabla u_0  - \e v_{1,\e}^{bl} - \e^2 \varUpsilon^\e \nabla^2 u_0 - \e^2 v_{2,\e}^{bl}.
	\end{equation*}
	A direct calculation as \cite[Theorem 9.1]{ShenZhuge16} shows that
	\begin{equation}
	\left\{
	\begin{aligned}
	\cL_\e w_\e &= \e^2 \frac{\partial}{\partial x_i} \bigg( f^\e_{ijk\ell} \frac{\partial^3 u_0}{\partial x_j \partial x_k \partial x_\ell}  \bigg)  &\quad & \txt{in } \Omega, \\
	w_\e &= 0 &\quad & \txt{on } \partial\Omega,
	\end{aligned}
	\right.
	\end{equation}
	where $f^\e_{ijk\ell}$ are some bounded periodic functions constructed in terms of $A, \chi$ and $\varUpsilon$. Therefore, we obtain (\ref{est.ueH2}) by the energy estimate.
\end{proof}

We point out that the key in (\ref{est.ueH2}) is to understand the asymptotic behavior of $\nabla v_{1,\e}^{bl}$ as $\e \to 0$. Since the Dirichlet boundary data of $v_{1,\e}^{bl}$ is rapidly oscillating on $\partial \Omega$, $\nabla v_{1,\e}^{bl}$ tends to blow up near the boundary as $\e$ approaching zero. Even the homogenized solution $v^{bl}_{1,0}$ lacks some sort of regularity (say, $H^2(\Omega;R^m)$), because as far as we know, the boundary data $v^{bl}_{1,0}|_{\partial\Omega} = f^*$ is only in $W^{1,p}(\partial\Omega;\R^m)$ for $p<\infty$. Nevertheless, as it has been proved in Theorem \ref{thm_blL2} that $v_{1,\e}^{bl}$ converges to $v_{1,0}^{bl}$ in $L^2(\Omega;\R^m)$ with a rate of $O(\e^{\frac{1}{2}-})$, in the following two lemmas, we are able to prove a general result that may handle this situation.
\begin{lemma}\label{lem.dD2.v0}
	Assume $v_0 \in H^1(\Omega;\R^m)$ and $\cL_0 v_0 \in L^2_{\txt{loc}}(\Omega;\R^m)$. Then
	\begin{equation*}
	\norm{\delta \nabla^2 v_0 }_{L^2(\Omega)} \le C\norm{\nabla v_0}_{L^2(\Omega)} + C\norm{\delta \cL_0 v_0 }_{L^2(\Omega)},
	\end{equation*}
	where $\delta(x) = \txt{dist}(x,\partial\Omega)$.
\end{lemma}
\begin{proof}
	This follows from the Caccioppoli's inequality. Actually, by setting $f=\cL_0 v_0$ and applying $\nabla$ to both sides, we have
	\begin{equation*}
	\cL_0 (\nabla v_0) = \nabla f.
	\end{equation*}
	For any $x\in \Omega$, the interior Caccioppoli's inequality implies that
	\begin{equation*}
	\int_{B(x,\delta(x)/4)} |\nabla^2 v_0|^2 \le \frac{C}{\delta(x)^2} \int_{B(x,\delta(x)/2)} |\nabla v_0|^2 + C\int_{B(x,\delta(x)/2)}|f|^2.
	\end{equation*}
	Observe that $\delta(y) \approx \delta(x)$ for any $y\in B(x,\delta(x)/2)$. The above inequality implies
	\begin{equation}\label{est.weightBall}
	\int_{B(x,\delta(x)/4)} \delta^2 |\nabla^2 v_0|^2 \le C \int_{B(x,\delta(x)/2)} |\nabla v_0|^2 + C\int_{B(x,\delta(x)/2)}\delta^2 |f|^2.
	\end{equation}
	Finally, we can cover $\Omega$ by a sequence of balls with finite overlaps, such that the sizes of these balls are comparable to the distance from boundary. The lemma follows then from (\ref{est.weightBall}).
\end{proof}
\begin{lemma}\label{lem.delta.H1}
	For any $v_\e, v_0 \in H^1(\Omega;\R^m)$ satisfying $\cL_\e v_\e = \cL_0 v_0 \in L^2(\Omega;\R^m)$, we have
	\begin{align}
		\begin{aligned}\label{est.delta.H1}
			& \norm{\delta( \nabla v_\e - \nabla v_0 - \nabla \chi^\e \nabla v_0)}_{L^2(\Omega)} \\
			& \qquad \le C\e \big(\norm{\nabla v_0}_{L^2(\Omega)} + \norm{\delta \cL_0 v_0 }_{L^2(\Omega)}\big) + C\norm{v_\e - v_0}_{L^2(\Omega)}.
		\end{aligned}
	\end{align}
\end{lemma}
\begin{proof}
	Define $w_\e = v_\e - v_0 - \e \chi^\e \nabla v_0$. Then a direct computation shows that
	\begin{equation*}
	\cL_\e w_\e = -\txt{div} ((\widehat{A} - A^\e - A^\e \nabla \chi^\e) \nabla v_0) + \txt{div} (\e \chi^\e \nabla^2 v_0).
	\end{equation*}
	By (\ref{eq.Bijk}) and a standard technique (see, e.g., \cite[Chapter 2.1]{ShenBook}), we have
	\begin{equation*}
	\cL_\e w_\e = -\txt{div} (\e B^\e \nabla^2 v_0) + \txt{div} (\e \chi^\e \nabla^2 v_0),
	\end{equation*}
	where $B = (b_{ijk})$ is defined by (\ref{eq.Bijk}). Observe that $\delta^2 w_\e \in H^1_0(\Omega;\R^m)$. Integrating the above equation against $\delta^2 w_\e$, we obtain
	\begin{equation*}
		\int_{\Omega} A^\e \nabla w_\e \cdot \nabla (\delta^2 w_\e) = \e \int_{\Omega} B^\e \nabla^2 v_0 \cdot \nabla(\delta^2 w_\e) - \e \int_{\Omega} \chi^\e\nabla^2 v_0\cdot \nabla(\delta^2 w_\e).
	\end{equation*}
	Using the fact $\norm{\nabla \delta}_{L^\infty(\Omega)} \le 1$ and the ellipticity condition, we have
	\begin{align}
		\begin{aligned}\label{est.Ddwe}
			\norm{\nabla (\delta  w_\e ) }_{L^2(\Omega)}^2 & \le C\int_{\Omega} A^\e \nabla(\delta w_\e) \cdot \nabla (\delta w_\e) \\
			& \le C\int_{\Omega} A^\e \nabla w_\e \cdot \nabla (\delta^2 w_\e) + C\norm{w_\e}_{L^2(\Omega)}^2\\
			& \le C\e \int_{\Omega} B^\e \nabla^2 v_0 \cdot \nabla(\delta^2 w_\e)  \\
			& \qquad - C \e \int_{\Omega} \chi^\e\nabla^2 v_0\cdot \nabla(\delta^2 w_\e) + C\norm{w_\e}_{L^2(\Omega)}^2.
		\end{aligned}
	\end{align}
	Now, it follows from the Cauchy-Schwarz inequality that
	\begin{align*}
		& \bigg|C\e \int_{\Omega} B^\e \nabla^2 v_0 \cdot \nabla(\delta^2 w_\e) \bigg| +  \bigg|C\e \int_{\Omega} \chi^\e \nabla^2 v_0 \cdot \nabla(\delta^2 w_\e) \bigg| \\
		& \quad \le C\e^2 \norm{\delta \nabla^2 v_0}_{L^2(\Omega)}^2 +  C\norm{w_\e}_{L^2(\Omega)}^2 + \frac{1}{2} \norm{\nabla (\delta  w_\e)}_{L^2(\Omega)}^2.
	\end{align*}
	Substituting this into (\ref{est.Ddwe}) and using Lemma \ref{lem.dD2.v0}, we obtain
	\begin{align*}
		\norm{\nabla(\delta w_\e)}_{L^2(\Omega)} & \le C\e \norm{\delta \nabla^2 v_0}_{L^2(\Omega)} + C\norm{w_\e}_{L^2(\Omega)}^2\\
		& \le C\e \big(\norm{\nabla v_0}_{L^2(\Omega)} + \norm{\delta \cL_0 v_0 }_{L^2(\Omega)}\big) + C\norm{v_\e - v_0}_{L^2(\Omega)}.
	\end{align*}
	The desired estimate follows by observing $\nabla(\delta w_\e) = \delta \nabla w_\e + \nabla \delta w_\e$ and that the second term is bounded by the right-hand side of (\ref{est.delta.H1}).
\end{proof}

The above lemma can be applied to find an interior expansion for $\nabla T_\e (\theta \phi_0 + \lambda_0 \Psi^{bl}\phi_0 )$ with an error of $O(\e)$. 
\begin{corollary}\label{cor.ve}
	Let $\phi_0$ be as before and Let $v_\e = T_\e (\theta \phi_0 + \lambda_0 \Psi^{bl}\phi_0 ) $ and $v_0 = T_0 (\theta \phi_0 + \lambda_0 \Psi^{bl}\phi_0 )$. Then
	\begin{equation*}
	\norm{\delta (\nabla v_\e - \nabla v_0 - \nabla\chi^\e \nabla v_0) }_{L^2(\Omega)} \le C\e.
	\end{equation*}
\end{corollary}
\begin{proof}
	Note that (\ref{est_L2rate}) implies that 
	\begin{equation}\label{est.ve-v0}
	\norm{v_\e - v_0}_{L^2(\Omega)} \le C\e\norm{\theta \phi_0 + \lambda_0 \Psi^{bl}\phi_0 }_{L^2(\Omega)}
	\end{equation}
	It follows from Lemma \ref{lem.delta.H1}, (\ref{est.ve-v0}) and the fact $\cL_0 v_0 = \theta \phi_0 + \lambda_0 \Psi^{bl}\phi_0$ that
	\begin{equation*}
	\norm{\delta (\nabla v_\e - \nabla v_0 - \nabla\chi^\e \nabla v_0) }_{L^2(\Omega)} \le C\e \norm{\theta \phi_0 + \lambda_0 \Psi^{bl}\phi_0 }_{L^2(\Omega)}.
	\end{equation*}
	The desired estimate follows readily.
\end{proof}

Now, it is sufficient to derive the first-order expansion for the leading term of (\ref{eq.Dphie}), $\nabla T_\e\phi_0$. Let $u_\e = T_\e \phi_0$ and $u_0 = T_0 \phi_0$. Observe that Theorem \ref{thm.ue2nd} gives
\begin{align}
	\begin{aligned}\label{eq.Due.2nd}
		\nabla u_\e & = (\nabla u_0 + \nabla \chi^\e \nabla u_0) + \e (\chi^\e \nabla^2 u_0 + \nabla \varUpsilon^\e \nabla^2 u_0) \\
		& \qquad + \e \nabla v^{bl}_{1,\e} + \e^2 \nabla v^{bl}_{2,\e} + O(\e^2).
	\end{aligned} 
\end{align}
By the energy estimate of (\ref{eq_bl2}), we have $\norm{\nabla v^{bl}_{2,\e}}_{L^2(\Omega)} \le C\e^{-1/2}$ and thus $\norm{\e^2 \nabla v^{bl}_{2,\e}}_{L^2(\Omega)} \le C\e^{3/2}$, which exactly is a higher-order error. Therefore, it suffices to consider the expansion of $\nabla v^{bl}_{1,\e}$, which can be handled by Lemma \ref{lem.delta.H1}.
\begin{corollary}\label{cor.ue}
	Let $u_\e = T_\e\phi_0$ and $u_0 = T_0 \phi_0$. Let $v^{bl}_{1,\e}$ and $v^{bl}_{1,0}$ be defined as (\ref{eq_bl}) and (\ref{eq.vblD}). Then
	\begin{equation*}
	\norm{\delta (\nabla v^{bl}_{1,\e} - \nabla v^{bl}_{1,0} - \nabla\chi^\e \nabla v^{bl}_{1,0}) }_{L^2(\Omega)} \le C\e^{\frac{1}{2}-}.
	\end{equation*}
\end{corollary}
\begin{proof}
	By the definition, one has $\cL_\e v^{bl}_{1,\e} = \cL_0 v^{bl}_{1,0} = F^*$, where
	\begin{equation*}
	F^* = -\bar{c}_{ijk} \frac{\partial^3 u_0}{\partial x_i \partial x_j \partial x_k}.
	\end{equation*}
	Since $\phi_0$ is the eigenfunction of $T_0$ corresponding to $\lambda_0^{-1}$, we know $u_0 = T_0\phi_0 = \lambda_0^{-1} \phi_0$. Thus, $\norm{F^*}_{L^2(\Omega)} \le C$, since $\phi_0$ is smooth in our setting. Also recall that $f^* \in W^{1,p}(\partial\Omega;\R^m)$ for any $p<\infty$, particularly for $p=2$. Now by Lemma \ref{lem.delta.H1}, we have
	\begin{align*}
		& \norm{\delta (\nabla v^{bl}_{1,\e} - \nabla v^{bl}_{1,0} - \nabla\chi^\e \nabla v^{bl}_{1,0}) }_{L^2(\Omega)} \\
		& \quad \le C\e \big(\norm{\nabla v^{bl}_{1,0}}_{L^2(\Omega)} + \norm{\delta \cL_0 v^{bl}_{1,0} }_{L^2(\Omega)}\big) + C\norm{v^{bl}_{1,\e} - v^{bl}_{1,0} }_{L^2(\Omega)}\\
		& \quad \le C\e \big(\norm{F^*}_{L^2(\Omega)} + \norm{f^*}_{H^1(\Omega)}\big) + C \e^{\frac{1}{2}-} \norm{\lambda_0^{-1} \phi_0 }_{W^{3,\infty}(\Omega)} \\
		& \quad\le C\e^{\frac{1}{2}-},
	\end{align*}
	where we have used Theorem \ref{thm_blL2} in the second inequality. This finishes the proof.
\end{proof}

\begin{proof}[Proof of Theorem \ref{thm.dDphie}]
	First note that $T_0 \phi_0 = \lambda_0^{-1} \phi_0$ and (\ref{eq_psiKK}) implies
	\begin{equation*}
	\lambda_0 T_0(\Psi^{bl} \phi_0) = \Psi^{bl}\phi_0 - (I-S_0)K^{bl} \phi_0.
	\end{equation*} 
	By Corollary \ref{cor.ve}, \ref{cor.ue} and (\ref{eq.Due.2nd}), one obtains
	\begin{equation*}
	\begin{aligned}
	& \delta \nabla T_\e (\theta \phi_0 + \lambda_0 \Psi^{bl}\phi_0 ) \\
	& = \delta(I+ \nabla \chi^\e) \nabla T_0 (\theta \phi_0 + \lambda_0 \Psi^{bl}\phi_0 ) + O(\e) \\
	& = \delta(I+ \nabla \chi^\e) (\theta \lambda_0^{-1} \nabla \phi_0 + \nabla \Psi^{bl}\phi_0 - \nabla(I-S_0)K^{bl}\phi_0 ) + O(\e)
	\end{aligned}
	\end{equation*}
	and
	\begin{equation*}
	\begin{aligned}
	\delta \nabla T_\e \phi_0  & = \delta(I+\nabla \chi^\e )\nabla T_0 \phi_0 + \e\delta (\chi^\e I + \nabla \varUpsilon^\e) \nabla^2 T_0 \phi_0 \\
	& \qquad + \e \delta (I + \nabla \chi^\e) \nabla v^{bl}_{1,0} + O(\e^{\frac{3}{2}-}) \\
	& = \delta \lambda_0^{-1} (I+\nabla \chi^\e )\nabla \phi_0 + \e\delta \lambda_0^{-1} (\chi^\e I + \nabla \varUpsilon^\e) \nabla^2 \phi_0 \\
	& \qquad + \e \delta \lambda_0^{-1} (I + \nabla \chi^\e) \nabla K^{bl}\phi_0 + O(\e^{\frac{3}{2}-}),
	\end{aligned}
	\end{equation*}
	where we also use the fact $v^{bl}_{1,0} = K^{bl}(T_0\phi_0) = \lambda_0^{-1} K^{bl} \phi_0$. Combining these together, we obtain
	\begin{equation*}
		\begin{aligned}
			\delta \nabla \phi_\e 
			& =  \delta(I+\nabla \chi^\e )\nabla \phi_0 \\ 
			& \quad + \e\delta \big[ (\chi^\e I + \nabla \varUpsilon^\e) \nabla^2 \phi_0 + (I+ \nabla \chi^\e) (\theta \lambda_0^{-1} \nabla \phi_0 + \nabla \Psi^{bl}\phi_0 + \nabla S_0K^{bl}\phi_0 )\big] \\
			& \quad + O(\e^{3/2}).
		\end{aligned}
	\end{equation*}
	Finally, by the definitions of $\theta$ and $S_0$, we have $\theta = -\lambda_0 \Ag{K^{bl}\phi_0, \phi_0}$ and $S_0 K^{bl}\phi_0 = \Ag{K^{bl}\phi_0, \phi_0} \phi_0$. It turns out that
	\begin{equation*}
	\theta \lambda_0^{-1} \nabla \phi_0 + \nabla S_0K^{bl}\phi_0 = 0.
	\end{equation*}
	This ends the proof as desired.
\end{proof}

\section{Neumann Problem}
In this section, we briefly introduce the Neumann eigenvalue problem, of which the results and proofs are parallel to those of Dirichlet problem, though the set-up will be slightly different due to the Neumann boundary condition. First, we introduce the operator $\widetilde{T}_\e$ defined on $ \dot{L}^2(\Omega; \R^m)$ by $\widetilde{T}_\e f = u_\e$, where $u_\e$ is the solution to
\begin{equation}\label{eq_NLeue}
\left\{
\begin{aligned}
\cL_\e u_\e &= f &\quad & \txt{in } \Omega, \\
\frac{\partial}{\partial \nu_\e} u_\e &= 0 &\quad & \txt{on } \partial\Omega.
\end{aligned}
\right.
\end{equation}
We also define $\widetilde{T}_0$ by $\widetilde{T}_0f = u_0$, where $u_0$ is the solution to
\begin{equation}\label{eq_NL0u0}
\left\{
\begin{aligned}
\cL_0 u_0 &= f &\quad & \txt{in } \Omega, \\
\frac{\partial}{\partial \nu_0} u_0 &= 0 &\quad & \txt{on } \partial\Omega.
\end{aligned}
\right.
\end{equation}
Note that the necessary condition of compatibility is $\int_{\Omega} f = 0$. 
Define
\begin{equation}
\dot{H}_{N,\e}^1 (\Omega; \R^m) = \bigg\{ f\in H^1(\Omega; \R^m): \int_{\Omega} f = 0, \frac{\partial}{\partial \nu_\e} f = 0 \bigg\},
\end{equation}
and
\begin{equation}
\dot{H}_{N,0}^1 (\Omega; \R^m) = \bigg\{ f\in H^1(\Omega; \R^m): \int_{\Omega} f = 0, \frac{\partial}{\partial \nu_0} f = 0 \bigg\}.
\end{equation}
Since $\widetilde{T}_\e$ are bounded linear operators from $\dot{L}^2(\Omega; \R^m)$ to $\dot{H}_{N,\e}^1 (\Omega; \R^m)$ uniformly in $\e>0$, $\{\widetilde{T}_\e: \e >0\}$ are collectively compact in $\dot{L}^2(\Omega; \R^m)$. Also, $\widetilde{T}_0$ is a bounded linear operator from $\dot{L}^2(\Omega; \R^m)$ to $\dot{H}_{N,0}^1 \cap H^2 (\Omega; \R^m)$ and therefore compact on $\dot{L}^2(\Omega; \R^m)$. Moreover, the classical homogenization theory (see, e.g., \cite[Chapter 6.1]{ShenBook}) implies
\begin{equation}\label{est.tTe-tT0}
\norm{\widetilde{T}_\e - \widetilde{T}_0}_{L^2(\Omega)} \le C \e.
\end{equation}

Let $\lambda_{\e, k}$ and $\lambda_{0,k}$ be the $k$th (in an increasing order) Neumann eigenvalue of $\cL_\e$ and $\cL_0$, respectively. Let $\phi_{\e,k}$ be the orthonormal eigenfunction of $\cL_\e$ corresponding to $\lambda_{\e, k}$ and $\phi_{0,k}$ be the orthonormal eigenfunction of $\cL_0$ corresponding to $\lambda_{0,k}$. In other words, one has
\begin{equation}
\cL_\e \phi_{\e,k} = \lambda_{\e, k} \phi_{\e,k}, \quad \phi_{\e,k} \in \dot{H}_{N,\e}^1 (\Omega; \R^m), \quad \norm{\phi_{\e,k}}_{L^2(\Omega)} = 1;
\end{equation}
and
\begin{equation}
\cL_0 \phi_{0,k} = \lambda_{0,k} \phi_{0,k}, \quad \phi_{0,k} \in \dot{H}_{N,0}^1 (\Omega; \R^m), \quad \norm{\phi_{0,k}}_{L^2(\Omega)} = 1.
\end{equation}
Recall that (\ref{est.tTe-tT0}) implies that $|\lambda_{\e,k} - \lambda_{0,k}| \le C_k \e$ for each $k\ge 1$.

In the following context, we let $\lambda_0 = \lambda_{0,L} = \lambda_{0,L+1} = \cdots =\lambda_{0, L+M-1}$ be a Neumann eigenvalue of $\cL_0$ with multiplicity $M\ge 1$. Let $\widetilde{S}_0$ be the spectral projection on the eigenspace of $\cL_0$ corresponding to $\lambda_0$, i.e., for any $f\in L^2(\Omega;\R^m)$,
\begin{equation*}
\widetilde{S}_0 f = \Ag{f,\phi_{0,L+j}}\phi_{0,L+j}.
\end{equation*}
Similarly, we denote by $\widetilde{S}_\e$ the spectral projection on the eigenspace of $\cL_\e$ corresponding to $\{\lambda_{\e,L+j}:0\le j\le M-1\}$, i.e.,
\begin{equation*}
\widetilde{S}_\e f = \Ag{f,\phi_{\e,L+j}}\phi_{\e,L+j}.
\end{equation*}

As Lemma \ref{lem_TeT} in Dirichlet problem, we also have the following lemma which is crucial to study the first-order expansion of Neumann eigenvalues and spectral projections.
\begin{lemma}\label{lem_NTeT}
	Let $\widetilde{K}^{bl}$ be the operator defined by (\ref{def_tK1}). Then,
	\begin{equation*}
	\norm{\widetilde{T}_\e - \widetilde{T}_0 - \e \chi^\e \nabla \widetilde{T}_0 - \e \widetilde{K}^{bl}   \widetilde{T}_0}_{\mathcal{R}(\widetilde{S}_0) \to L^2(\Omega)} \le C\e^{\frac{3}{2}-}.
	\end{equation*}
\end{lemma}

Lemma \ref{lem_NTeT} is a straightforward corollary of Theorem \ref{thm_NL2NextOrder}. With this lemma and the previous settings adapted to Neumann problem, we can mimic the argument of Dirichlet problem and obtain exactly the same results. We state the results below without proofs. 

\begin{theorem}\label{thm_Nlambda}
	Let $A$ and $\Omega$ satisfy the same assumptions as Theorem \ref{thm_lambda}. Let $\lambda_0, \lambda_{\e,L+j}$  $(0\le j\le M-1)$ be the Neumann eigenvalues defined previously. Then there exist a constant $\theta$ independent of $\e$ such that for sufficiently small $\e > 0$
	\begin{equation*}
	|\bar{\lambda}_\e - \lambda_0 - \e \theta| = C\e^{\frac{3}{2}-},
	\end{equation*}
	where $\bar{\lambda}_\e = M^{-1} \sum_{j=0}^{M-1} \lambda_{\e,L+j},\ \theta = -\lambda_0 M^{-1} \Ag{ \widetilde{K}^{bl}\phi_{0,L+j}, \phi_{0,L+j}}$ and $C$ depends only on $\lambda_0,A$ and $\Omega$.
\end{theorem}

\begin{theorem}\label{thm.NSe}
	Let $A$ and $\Omega$ satisfy the same assumptions as Theorem \ref{thm_lambda} and let $\widetilde{S}_\e$ and $\widetilde{S}_0$ be the spectral projections defined above. Then,
	\begin{equation}\label{est.Se-RS0N}
	\norm{\widetilde{S}_\e - \widetilde{S}_0 - \e (\chi^\e \nabla + \widetilde{\Psi}^{bl})   \widetilde{S}_0} _{\mathcal{R}(\widetilde{S}_0) \to L^2(\Omega)} \le C\e^{\frac{3}{2}-},
	\end{equation}
	and
	\begin{equation*}
	\norm{\widetilde{S}_\e - \widetilde{S}_0 - \e (\chi^\e \nabla + \widetilde{\Psi}^{bl})   \widetilde{S}_0 - \e \widetilde{S}_0   (\chi^\e \nabla + \widetilde{\Psi}^{bl})^*}_{L^2(\Omega) \to L^2(\Omega)} \le C\e^{\frac{3}{2}-},
	\end{equation*}
	where the bounded linear operator $\widetilde{\Psi}^{bl}:\mathcal{R}(\widetilde{S}_0) \mapsto \mathcal{R}(\widetilde{S}_0)^\perp$ is given by
	\begin{equation}\label{eq.tPsi}
	\widetilde{\Psi}^{bl} g = \lambda_0^{-1} (\lambda_0^{-1}-\widetilde{T}_0)^{-1} (I- \widetilde{S}_0 )\widetilde{K}^{bl} g.
	\end{equation}
\end{theorem}

\begin{corollary}\label{cor.NeuSim}
	Let $A$ and $\Omega$ satisfy the same assumptions as Theorem \ref{thm_lambda}. Let $\lambda_0 = \lambda_{0,L}$ be a simple Neumann eigenvalue of $\cL_0$ with eigenfunction $\phi_0 = \phi_{0,L}$ and $\lambda_\e = \lambda_{\e,L}$ be the Neumann eigenvalue of $\cL_\e$ with eigenfunction $\phi_\e = \phi_{\e,L}$. Then for $\e>0$ sufficiently small,
	\begin{equation*}
	\norm{\phi_\e - \phi_0 - \e \chi^\e \nabla \phi_0 - \e \widetilde{\Psi}^{bl} \phi_0} _{L^2(\Omega)} \le C \e^{\frac{3}{2}-},
	\end{equation*}
	where the operator $\widetilde{\Psi}^{bl}$ is given by (\ref{eq.tPsi}).
\end{corollary}

\begin{theorem}\label{thm.dDphieN}
	Under the same assumptions as Corollary \ref{cor.NeuSim}, we have the following expansion of $\delta \nabla \phi_\e$ in the sense of $L^2(\Omega;\R^m)$,
	\begin{equation*}
		\begin{aligned}
			\delta \nabla \phi_\e 
			& =  \delta(I+\nabla \chi^\e )\nabla \phi_0 \\ 
			& \quad + \e\delta \big[ (\chi^\e I + \nabla \varUpsilon^\e) \nabla^2 \phi_0 + (I+ \nabla \chi^\e) \nabla \widetilde{\Psi}^{bl}\phi_0 \big]  + O(\e^{\frac{3}{2}-}).
		\end{aligned}
	\end{equation*}	
\end{theorem}

\subsection*{Acknowledgment} The author would like to thank Professor Zhongwei Shen for bringing the eigenvalue problem to him and helpful discussions.

\bibliographystyle{amsplain}
\bibliography{mybib}

\def\cprime{$'$}
\providecommand{\bysame}{\leavevmode\hbox to3em{\hrulefill}\thinspace}
\providecommand{\MR}{\relax\ifhmode\unskip\space\fi MR }
\providecommand{\MRhref}[2]{%
  \href{http://www.ams.org/mathscinet-getitem?mr=#1}{#2}
}
\providecommand{\href}[2]{#2}
\begin{thebibliography}{10}

\bibitem{AC98}
Gr\'egoire Allaire and Carlos Conca, \emph{Bloch wave homogenization and
  spectral asymptotic analysis}, J. Math. Pures Appl. (9) \textbf{77} (1998),
  no.~2, 153--208. \MR{1614641}

\bibitem{AKMP16}
Scott Armstrong, Tuomo Kuusi, Jean-Christophe Mourrat, and Christophe Prange,
  \emph{Quantitative analysis of boundary layers in periodic homogenization},
  Arch. Ration. Mech. Anal. \textbf{226} (2017), no.~2, 695--741. \MR{3687879}

\bibitem{GM11}
David G\'erard-Varet and Nader Masmoudi, \emph{Homogenization in polygonal
  domains}, J. Eur. Math. Soc. (JEMS) \textbf{13} (2011), no.~5, 1477--1503.
  \MR{2825170}

\bibitem{GM12}
\bysame, \emph{Homogenization and boundary layers}, Acta Math. \textbf{209}
  (2012), no.~1, 133--178. \MR{2979511}

\bibitem{KS13}
Carlos Kenig, Fanghua Lin, and Zhongwei Shen, \emph{Estimates of eigenvalues
  and eigenfunctions in periodic homogenization}, J. Eur. Math. Soc. (JEMS)
  \textbf{15} (2013), no.~5, 1901--1925. \MR{3082248}

\bibitem{Kesavan1}
Srinivasan Kesavan, \emph{Homogenization of elliptic eigenvalue problems. {I}},
  Appl. Math. Optim. \textbf{5} (1979), no.~2, 153--167. \MR{533617}

\bibitem{Kesavan2}
\bysame, \emph{Homogenization of elliptic eigenvalue problems. {II}}, Appl.
  Math. Optim. \textbf{5} (1979), no.~3, 197--216. \MR{546068}

\bibitem{MV97}
Shari Moskow and Michael Vogelius, \emph{First-order corrections to the
  homogenised eigenvalues of a periodic composite medium. {A} convergence
  proof}, Proc. Roy. Soc. Edinburgh Sect. A \textbf{127} (1997), no.~6,
  1263--1299. \MR{1489436}

\bibitem{Os75}
John~E. Osborn, \emph{Spectral approximation for compact operators}, Math.
  Comput. \textbf{29} (1975), 712--725. \MR{0383117}

\bibitem{Prange13}
Christophe Prange, \emph{First-order expansion for the {D}irichlet eigenvalues
  of an elliptic system with oscillating coefficients}, Asymptot. Anal.
  \textbf{83} (2013), no.~3, 207--235. \MR{3112850}

\bibitem{SV93}
Fadil Santosa and Michael Vogelius, \emph{First-order corrections to the
  homogenized eigenvalues of a periodic composite medium}, SIAM J. Appl. Math.
  \textbf{53} (1993), no.~6, 1636--1668. \MR{1247172}

\bibitem{ShenBook}
Zhongwei Shen, \emph{Lectures on periodic homogenization of elliptic systems},
  arXiv: 1710.11257 (2017).

\bibitem{ShenZhuge16}
Zhongwei Shen and Jinping Zhuge, \emph{Boundary layers in periodic
  homogenization of {N}eumann problems}, Comm. Pure Appl. Math. (To appear).

\bibitem{ShenZhuge17}
\bysame, \emph{Regularity of homogenized boundary data in periodic
  homogenization of elliptic systems}, J. Eur. Math. Soc. (JEMS) (To appear).

\bibitem{Zhuge17}
Jinping Zhuge, \emph{Homogenization and boundary layers in domains of finite
  type}, Comm. Partial Differential Equations (To appear).

\end{thebibliography}
\end{document}